\documentclass{jfl}

\usepackage[Symbol]{upgreek}

\newcommand{\bbold}{\mathbb}

\usepackage{tikz}
\usetikzlibrary{decorations.pathreplacing}

\def\R { {\bbold R} }
\def\Q { {\bbold Q} }

\def\N { {\bbold N} }

\renewcommand\epsilon{\varepsilon}

\def \<{\langle}
\def \>{\rangle}

\def\((  {(\!(}
\def\)) {)\!)}
\def \k {{{\boldsymbol{k}}}}

\font\tiny=cmsy7

\def \RLE {\R \((  x^{-1}  \)) ^{\text {LE}}  }

\def \I{\operatorname{I}}

\def \i{{\boldsymbol{i}}}
\def \j{{\boldsymbol{j}}}

\def \Log{\operatorname{L}}

\def \upg{\upgamma}
\def \upl{\uplambda}
\def \Upl{\Uplambda}
\def \upo{\upomega}
\def \Upo{\Upomega}
\def \ome{\omega}

\theoremstyle{theorem}
\newtheorem{prop}[theorem]{Proposition}
\newtheorem{cor}[theorem]{Corollary}

\newtheorem*{conjecture*}{Conjecture}
\newtheorem*{Tconjecture}{$\T$-Conjecture}

\newtheorem*{refinedTconjecture}{Refined $\T$-Conjecture}
\newtheorem*{optimisticconjecture}{Optimistic Conjecture}

\theoremstyle{definition}

\theoremstyle{remark}
\newtheorem*{example*}{Example}

\newtheorem*{notation*}{Notation}

\newtheorem*{remark*}{Remark}

\newtheorem*{question*}{Question}

\numberwithin{equation}{section}

\newcommand{\abs}[1]{\lvert#1\rvert}

\def \T{\mathbb{T}}
\def \Tas{\mathbb{T}^{\operatorname{as}}}
\def \st{\operatorname{st}}

\def \Ca{\mathcal{C}}
\def \<{\langle}
\def \>{\rangle}
\def \csi{\operatorname{si}}

\renewcommand\leq{\leqslant}
\renewcommand\geq{\geqslant}
\renewcommand\preceq{\preccurlyeq}
\renewcommand\succeq{\succcurlyeq}
\renewcommand\le{\leq}
\renewcommand\ge{\geq}

\DeclareSymbolFont{imag@m}{OT1}{cmr}{m}{ui}
\DeclareMathSymbol{\imag}{\mathord}{imag@m}{105}

\DeclareMathSymbol{\ex}{\mathord}{imag@m}{101}

\DeclareFontFamily{OMS}{smallo}{}
\DeclareFontShape{OMS}{smallo}{m}{n}{<->s*[.65]cmsy10}{}
\DeclareSymbolFont{smallo@m}{OMS}{smallo}{m}{n}
\DeclareMathSymbol{\smallo}{\mathord}{smallo@m}{79}

\DeclareFontFamily{OMS}{largerdot}{}
\DeclareFontShape{OMS}{largerdot}{m}{n}{<->s*[.8]cmsy10}{}
\DeclareSymbolFont{largerdot@m}{OMS}{largerdot}{m}{n}
\DeclareMathSymbol{\largerdot}{\mathord}{largerdot@m}{15}

\DeclareFontFamily{U}{fsy}{}
\DeclareFontShape{U}{fsy}{m}{n}{<->s*[1.0]psyr}{}
\DeclareSymbolFont{der@m}{U}{fsy}{m}{n}
\DeclareMathSymbol{\der}{\mathord}{der@m}{182}

\received{}
\printed{}

\begin{document}
\title{Towards a Model Theory for Transseries}

\author{Matthias~Aschenbrenner}
\address{University of California, Los Angeles\\
Los Angeles, CA 90955\\
U.S.A.}
\email {matthias@math.ucla.edu}

\author{Lou~van~den~Dries}
\address{University of Illinois at Urbana-Champaign\\
Urbana, IL 61801\\
U.S.A.}
\email{vddries@math.uiuc.edu}

\author{Joris~van~der~Hoeven}
\address{\'Ecole Polytechnique\\
91128 Palaiseau Cedex\\
France}
\email{vdhoeven@lix.polytechnique.fr}

\dedicatory{For Anand Pillay, on his 60th birthday.}

\thanks{Aschenbrenner was partially supported by NSF Grants DMS-0556197 and DMS-0969642. 
The work of van der Hoeven has been partly supported by the French ANR-09-JCJC-0098-01 MaGiX project,
and by the Digiteo 2009-36HD grant of the R\'egion Ile-de-France.
All three authors thank the Fields Institute, Toronto, for its hospitality during the Thematic Program on o-minimal Structures and Real Analytic Geometry (2009). Some of the work reported on here was done on that occasion.}

\begin{abstract}
The differential field of transseries extends the field of real
Laurent series, and occurs in various contexts: asymptotic expansions,
analytic vector fields, o-minimal structures, to name a few. 
We give an overview of the algebraic and
model-theoretic aspects of this differential field, and report on our 
efforts to understand its elementary theory.
\end{abstract}

\keywords{Transseries; Hardy fields; differential fields; model completeness; NIP.}

\maketitle


\section*{Introduction}

\noindent
We shall describe a fascinating mathematical object, the differential
field $\T$ of transseries. It is an ordered field extension of $\R$
and is a kind of universal domain for asymptotic real differential algebra.
In the context of this paper, a {\em transseries\/} is what is called a 
{\em logarithmic-exponential series\/} or LE-series in~\cite{DMM}. 
Here is the main problem that we have been pursuing, intermittently, for more than 15 years.

\begin{conjecture*}
The theory of the ordered differential field $\T$ is model complete, 
and is the model companion of the theory of $H$-fields with
small derivation.
\end{conjecture*}

With slow progress during many years, 
our understanding of the situation has recently increased at a faster rate, 
and this is what we want to report on. In Section~\ref{sec:transseries} we give an informal
description of $\T$, in Section~\ref{sec:T-conjecture} we give some evidence for
the conjecture and indicate some plausible consequences. In Section~\ref{sec:H-fields}
we define $H$-fields, and explain their expected role in the story.
Section~\ref{sec:new results} describes our recent partial results towards the 
conjecture, obtained since the publication of the survey~\cite{AvdD3}. (A full account is in preparation, and of course we hope to 
finish it with a proof of the conjecture.) Section~\ref{sec:qf definability} proves 
quantifier-free versions of the conjectural induced structure on the constant field $\R$ of $\T$, of the asymptotic o-minimality of $\T$, and of $\T$
having NIP. In the last Section~\ref{sec:obstructions} we discuss 
what might be the right 
primitives to eliminate quantifiers for $\T$; this amounts to a 
strong form of the  above conjecture. 

\medskip\noindent
This paper is mainly expository and programmatic in nature, and occasionally
speculative. It is meant to be readable with only a 
rudimentary knowledge of model theory, valuations, and differential fields, and 
elaborates on talks by us on various recent 
occasions, in particular by the second-named author at the meeting in Ol\'eron.
For more background on the material in Sections~\ref{sec:transseries}--\ref{sec:H-fields} (for example, on Hardy fields) see~\cite{AvdD3}, which can serve as a companion to the present paper.

\subsection*{Conventions.} Throughout, $m$, $n$ range over $\N=\{0,1,2,\dots\}$.
For a field $K$ we let $K^{\times}=K\setminus \{0\}$ be its multiplicative 
group. By a {\em Hahn field\/} we mean a field
$\k((t^{\Gamma}))$ of generalized power series, with coefficients in the field $\k$ and
exponents in a non-trivial ordered abelian group $\Gamma$, and we view it as 
a {\em valued\/} field in the usual way\footnote{Strictly speaking, any valued field isomorphic to such a
generalized power series field is also considered as a Hahn field 
in this paper.}.  
By {\em differential field\/} we mean a field $K$ of characteristic zero
equipped with a derivation $\der\colon K \to K$. 
In our work the operation of taking the 
{\em logarithmic\/} derivative is just as basic as the derivation
itself, and so we introduce a special notation: $y^\dagger:=  y'/y$ denotes the
logarithmic derivative of a non-zero $y$ in a differential field. Thus
$(yz)^\dagger = y^\dagger + z^\dagger$ for non-zero~$y$,~$z$ in a differential field.
Given a differential field
$K$ and an element $a$ in a differential field extension of $K$ we
let $K\<a\>$ be the differential field generated by $a$ over $K$.
An {\em ordered 
differential field\/} is a differential field equipped with an ordering in 
the usual sense 
of ``ordered field.'' A {\em valued differential field\/} 
is a differential field equipped with a (Krull) valuation that is trivial 
on its prime subfield $\Q$. The term {\em pc-sequence\/} abbreviates {\em pseudo-cauchy sequence}. 

\section{Transseries}\label{sec:transseries}

\noindent
The ordered differential field $\mathbb T$ of transseries arises as a natural remedy for certain shortcomings of the ordered differential field of formal Laurent series.

\subsection{Laurent series.}   
Recall that the field $\R \(( x^{-1}  \)) $ of formal
Laurent series in powers of $x^{-1}$ over $\R$ consists of all series of the form
$$f(x)\ =\ \underbrace{ a_nx^n + a_{n-1}x^{n-1} + \cdots + a_1x}_{\text
{infinite part of $f$}} + a_0 +  \underbrace{a_{-1}x^{-1} +
a_{-2}x^{-2}+ \cdots}_{\text{infinitesimal part of $f$}}$$
with real coefficients $a_n, a_{n-1},\dots$
We order $\R \(( x^{-1}  \)) $ by requiring $x > \R$, and make it a differential
field by requiring $x' = 1$ and differentiating termwise. 

\medskip
\noindent
The ordered differential field $\R \(( x^{-1}  \)) $ is too small for many purposes:
\begin{itemize}
\item  $x^{-1}$ has no antiderivative $\log x$  in $ \R\(( x^{-1} \)) $;
\item  there is no reasonable exponentiation $f\mapsto \exp(f)$. 
\end{itemize}
Here ``reasonable'' means that it extends real
exponentiation and preserves its key properties: the map $f\mapsto \exp(f)$ 
should be an isomorphism from the ordered additive group of $\R\(( x^{-1} \)) $ 
onto its ordered multiplicative group of positive elements, and $\exp(x)> x^n$
for all $n$ in view of $x>\R$.  
Note that exponentiation does make sense for 
the {\em finite\/}  elements of $\R \(( x^{-1} \)) $: 
\begin{align*}  &\exp (a_0 + a_{-1}x^{-1} +
a_{-2}x^{-2}+ \cdots) \\&= \ex^{a_0}\sum_{n=0}^{\infty}\frac{1}{n!}
(a_{-1}x^{-1} + a_{-2}x^{-2} + \cdots)^n\\
&= \ex^{a_0}(1 + b_1x^{-1} +b_2x^{-2} + \cdots) \quad\text{for suitable $b_1,b_2,\ldots\in\R$.} \end{align*}
The main {\em model-theoretic\/} defect of $\R \(( x^{-1}  \)) $ 
as a differential field is that it defines the 
subset $\mathbb Z$; see \cite[proof of Proposition~3.3,~(i)]{DL}.  Thus it has no ``tame'' 
model-theoretic features. (In contrast, $\R \(( x^{-1}  \)) $  viewed as just 
a field is decidable by the work of Ax and Kochen~\cite{AK66}.)

\subsection{Transseries.}\label{subsec:transseries}
To remove these defects we extend  $ \R \((
x^{-1} \)) $ to an ordered differential field $\T$ of 
{\em transseries\/}:
series of  {\em transmonomials\/}  (or logarithmic-exponential monomials)  arranged
from left to right in decreasing order, each multiplied by a real
coefficient, for example 
$$ \ex^{\ex^x} - 3\ex^{x^2} + 5x^{\sqrt{2}} - (\log x)^{\pi} + 1 +
x^{-1} + x^{-2} + x^{-3} + \cdots + \ex^{-x} + x^{-1}\ex^{-x} \text{  .}$$ 
The reversed order type of the set of transmonomials that occur in a
given trans\-series can be any countable ordinal. (For the series displayed it
is $\omega + 2$.)
As with $ \R \(( x^{-1} \)) $, the natural derivation of $\T$ is given by 
termwise differentiation of such series, and in the natural ordering on $\T$, a non-zero transseries is positive iff its leading (``leftmost'') 
coefficient is positive. 

\medskip
\noindent
Transseries occur in solving implicit equations of the
form $P(x,y,\ex^x,\ex^y)=0$ for~$y$ as $x \to +\infty$, where $P$ is a
polynomial in four variables over $\R$. 
More generally, transseries occur as asymptotic expansions of functions definable in o-minimal expansions of the real field; see \cite{AvdD3} for more on this.
Transseries also arise as formal solutions
to algebraic differential equations and in many other ways.  
For example, the Stirling expansion for the Gamma function is a 
(very simple) transseries.

\medskip
\noindent
The terminology ``transseries'' is due 
to  \'Ecalle who introduced $\T$ in his
solution of Dulac's Problem: a polynomial vector field in the plane 
can only have finitely many limit cycles; see~\cite{E}. (This is related to 
Hilbert's 16th Problem.)
Independently, $\T$ was also defined by Dahn and G\"{o}ring in \cite{DG}, 
in connection with Tarski's problem on the real exponential field, and studied
as such in \cite{DMM}, in the aftermath of Wilkie's famous theorem \cite{Wilkie}.   
(Discussions of the history of transseries are in \cite{vdH,Ressayre2}.)

\medskip
\noindent
Transseries are added and multiplied in the usual way and form a ring
$\T$, and this ring comes equipped with several other natural operations.
Here are a few, each accompanied by simple examples and relevant facts 
about $\T$:

\subsubsection*{Taking the multiplicative inverse.}
Each non-zero $f\in\T$ has a multiplicative inverse in $\T$: for example,
\begin{align*} \frac{1}{x - x^2\ex^{-x}}\ =\ \frac{1}{x(1-x\ex^{-x})}\  &=\
x^{-1}(1+x\ex^{-x}+x^2\ex^{-2x} + \cdots )\\ 
 &=\ x^{-1}+ \ex^{-x}+x\ex^{-2x} + \cdots  \end{align*} 
As an ordered field, $\T$ is a real closed extension of $\R$. In particular, an algebraic closure of $\T$ is given by $\T[\imag]$ where $\imag^2=-1$.

\subsubsection*{Formal differentiation.} 
Each $f\in\T$ can be differentiated term by term, giving
a derivation $f\mapsto f'$ on the field $\T$. For example: 
$$(\ex^{-x}+\ex^{-x^2}+\ex^{-x^3}+\cdots)'\  =\ 
-(\ex^{-x}+2x\ex^{-x^2}+3x^2\ex^{-x^3}+\cdots).$$
The field of constants of this derivation is
$\{f\in \T: \  f'=0\}=\R$. 

\subsubsection*{Formal integration.} For each $f\in\T$ there is some $F\in\T$ (unique up to addition of a constant from $\R$) with $F'=f$: for example,
$$\int {{\ex^x}\over x}\,dx\  =\ \text{constant} + \sum_{n=0}^{\infty} n!x^{-1-n}\ex^x \quad \text{(diverges).}$$

\subsubsection*{Formal composition.} Given $f,g\in\T$ with $g>\R$, we can ``substitute $g$ for $x$ in $f$'' to obtain a transseries $f\circ g\in\T$. For example,
let $f(x)\ =\ x + \log x$ and $g(x) = x\log x$; writing $f(g(x))$ for 
$f\circ g$, we have
\begin{align*}  f(g(x))\ &=\ x\log x + \log (x\log x)\
            =\ x\log x + \log x + \log(\log x), \\
g(f(x))\ &=\ (x + \log x)\log(x+\log x) \\
           &=\ x\log x + (\log x)^2 + (x+\log x)\sum_{n=1}^{\infty}
\frac{(-1)^{n+1}}{n}\left( \frac{\log x}{x}\right)^n\\
          &=\ x\log x + (\log x)^2 + \log x + \sum_{n=1}^{\infty}
\frac{(-1)^{n+1}}{n(n+1)} \frac{(\log x)^{n+1}}{x^n}.
\end{align*}  
The Chain Rule holds: 
$$(f\circ g)'\ =\ (f'\circ g)\cdot g'\qquad\text{ for all $f,g\in\T$, $g>\R$.}$$

\subsubsection*{Compositional inversion.}
The set $\T^{>\R}:=\{f\in\T:f>\R\}$ of positive infinite transseries is closed under the composition operation $(f,g)\mapsto f\circ g$, and forms a group with identity element $x$. For example,
the transseries $g(x) = x\log x$ has a compositional inverse of the form
$$  \frac{x}{\log x}\cdot F\left(\frac{\log \log x}{\log x},
\frac{1}{\log x}\right)$$
where $F(X,Y)$ is an ordinary convergent power series in the two variables
$X$ and~$Y$ over $\R$ with constant term $1$. (This fact plays a certain role in the solution, using transseries, of a   problem of Hardy dating from 1911, obtained independently in \cite{DMM1} and \cite{vdH:PhD}; see~\cite{Ressayre2}.)

\subsubsection*{Exponentiation.} We have a canonical isomorphism 
$f\mapsto\exp(f)$, with inverse $g\mapsto\log(g)$, between the ordered additive group of $\T$ and the ordered multiplicative group~$\T^{>0}$; it extends the exponentiation of finite Laurent series described above. With $\sinh:=\frac{1}{2}\ex^x - \frac{1}{2}\ex^{-x}\in\T^{>0}$ (sinus hyperbolicus),
\begin{align*}
\exp(\sinh)\ 	&=\ \exp\left(\textstyle\frac{1}{2}\ex^x\right)\cdot\exp\left(-\textstyle\frac{1}{2}\ex^{-x}\right) \\
			&=\ \ex^{\frac{1}{2}\ex^x} \cdot \sum_{n=0}^\infty \frac{1}{n!}\left(-\textstyle\frac{1}{2}\ex^{-x}\right)^n\ 
			 =\ \sum_{n=0}^\infty \frac{(-1)^n}{n!2^n}\ex^{\frac{1}{2}\ex^x-nx},\\
\log(\sinh)\	&=\ 
\log\left(\textstyle\frac{\ex^x}{2}\left(1-\ex^{-2x}\right)\right)\ 
			 =\  x-\log 2 - \sum_{n=1}^\infty \frac{1}{n}\ex^{-2nx}.
\end{align*}
As an exponential ordered field, $\T$ is an elementary extension of the real 
exponential field~\cite{DMM1}, and thus model complete 
and o-minimal~\cite{Wilkie}.
The iterated exponentials 
$$\ex_0:=x,\  \ex_1:=\exp x,\  \ex_2:=\exp(\exp(x)),\ \dots$$ 
form an increasing cofinal sequence in the ordering of $\T$. Likewise, their formal 
compositional inverses  
$$\ell_0:=x,\ \ell_1:=\log x,\ \ell_2:=\log(\log(x)),\ \dots$$ 
form a decreasing coinitial sequence in  $\T^{>\R}$.

\medskip
\noindent
A precise construction of $\T$ is in \cite{DMM}, 
where it is denoted by $\RLE$. See also \cite{E}, \cite{Edgar} and \cite{vdH}
for other accounts. The {\em purely logarithmic transseries\/} are those which, 
informally speaking, do not involve exponentiation, and they make up an
intriguing differential subfield $\T_{\log}$ of $\T$ that has a very 
explicit definition:  
First, setting
$\ell_0:=x$ and $\ell_{m+1}=\log\ell_{m}$ yields the sequence $(\ell_m)$ 
of iterated logarithms of $x$. Next, 
let $\mathfrak L_m$ be the formal multiplicative group 
$$\ell_0^{\R}\cdots \ell_m^{\R}=\big\{\ell_0^{r_0}\cdots \ell_m^{r_m}:r_0,\dots,r_m\in\R\big\},$$ 
made into an ordered group
such that $\ell_0^{r_0}\cdots \ell_m^{r_m}>1$ 
iff the exponents $r_0, \dots, r_m$ are not all zero, and $r_i>0$ for the least $i$ with 
$r_i\ne 0$. Of course, if $m\le n$, then $\mathfrak L_m$ is naturally an ordered subgroup of $\mathfrak L_n$, and so we have a natural inclusion of Hahn fields 
$\R \(( \mathfrak L_m \))  \subseteq \R \(( \mathfrak L_n \)) $. We now have
$$ \T_{\log} = \bigcup_{n=0}^\infty \R\(( \mathfrak L_n \))   \qquad\text{(increasing union of differential subfields).}$$
It is straightforward to define $\log f\in \T_{\log}$ for $f\in \T^{>0}_{\log}$.

The inductive construction of $\T$ is more complicated, but
also yields $\T$ as a directed union of Hahn subfields, each of which is also
closed under the derivation.
Hahn fields themselves (as opposed to suitable directed unions of Hahn fields)
cannot be equipped with a reasonable exponential map; see 
\cite{KKS}.

Note that $\T_{\log}$ is a proper subfield of $\R\(( \mathfrak L \)) $, where $\mathfrak L:=\bigcup_{n=0}^\infty \mathfrak L_n$ (directed union of ordered multiplicative subgroups): for example, the series
$$\frac{1}{\ell_0^2} + \frac{1}{(\ell_0\ell_1)^2} + \cdots + \frac{1}{(\ell_0 \ell_1\cdots \ell_n)^2} + \cdots $$
lies in $\R\(( \mathfrak L \)) $, but not in $\T_{\log}$, and in fact, not even in $\T$. (This series will be important in Section~\ref{sec:new results} below; see also Theorem~\ref{thm:2ndorder}.)

\subsection{Analytic counterparts of $\T$.} 

Convergent series in $ \R \((
x^{-1} \)) $ define
germs of real analytic functions at infinity. This yields
an isomorphism of ordered differential fields
between the subfield of convergent series
in $ \R \(( x^{-1} \)) $ and a Hardy field. It would be desirable to
extend this to isomorphisms between larger differential subfields $T$ of 
$\T$ and Hardy fields $H$ which preserves as much structure as possible: 
the ordering, differentiation, and even integration and 
composition, whenever defined. 

However, if $T$ is sufficiently 
closed under integration (or solutions of other simple differential
equations), then it will contain 
divergent power series in~$x^{-1}$,
as well as more general divergent transseries. A major difficulty
is to give an analytic meaning to such transseries. In simple cases,
Borel summation provides a systematic device for doing this.
Borel's theory has been greatly extended by \'Ecalle, who introduced
a big subfield $\Tas$ of $\T$. The elements of $\Tas$ are
called {\em accelero-summable transseries\/}, and $\Tas$
is real closed, 
stable under differentiation, integration, composition, etc.
The analytic counterparts of accelero-summable transseries
are called {\em analysable functions\/},
and they appear naturally in \'Ecalle's proof of the Dulac Conjecture.
As a prelude to the $\T$-Conjecture in the next section, here are some sweeping statements\footnote{The English translation given here is ours; the original sentences are on p.~148. We
also used our notations $\T$ and $\Tas$ instead of \'Ecalle's $\R[[[x]]]$ and 
$\R\{\!\{\!\{x\}\!\}\!\}$.}  
from  \'Ecalle's book \cite{E}  on these notions, indicating that~$\T$ and its cousin $\Tas$ might be viewed as {\it universal domains}\/ for
asymptotic analysis. 

\medskip\noindent
{\em It seems \textup{[ \dots]} \textup{(}but I have not yet verified this in all generality\textup{)} that  $\Tas$ is closed under resolution of differential equations, or, more exactly, that if a differential equation has formal solutions in $\T$, then these solutions are automatically in $\Tas$.}

\medskip\noindent
{\em It seems \textup{[ \dots]} that the algebra $\Tas$ of accelero-summable transseries is 
truly the algebra-from-which-one-can-never-exit and that it marks an almost  
impassable horizon for ``ordered analysis.'' \textup{(}This  sector of analysis is in some sense ``orthogonal''  to harmonic analysis.\textup{)}}

\medskip\noindent
{\em This notion of analysable function represents probably the ultimate 
extension of the notion of \textup{(}real\textup{)} analytic function, and it seems inclusive 
and stable to a degree unheard of.}

\medskip\noindent
Accelero-summation requires a big machinery.
If we just try to construct isomorphisms
$T \rightarrow H$ which do not necessarily preserve composition but
do preserve the ordering and differentiation, then simpler arguments
with a more model-theoretic flavor can be used to prove
the following, from~\cite{vdH:hfsol}:

\begin{theorem}
Let $\T^{\operatorname{da}}\subseteq \T$ be the field of transseries 
that are differentially algebraic over $\R$. Then there is
an isomorphism of ordered differential fields between $\T^{\operatorname{da}}$
and some Hardy field. 
\end{theorem}

In~\cite{vdH:hfsol}, this follows from general theorems
about extending isomorphisms between suitable
differential subfields of $\T$ and Hardy fields. 

\section{The $\T$-Conjecture} \label{sec:T-conjecture}

\noindent
As explained above, the elementary theory of $\T$ as an
exponential field is understood, but $\T$ is far more interesting when 
viewed as a {\em differential\/} field.

\bigskip\noindent
\begin{center}
{\bf From now on we consider $\T$ as an ordered valued differential field.}
\end{center}

\medskip\noindent
\begin{Tconjecture}
$\T$ is model complete.
\end{Tconjecture}

Model completeness is fairly robust as to which first-order language is used,
but to be precise, we consider $\T$ here as an $\mathcal L$-structure, where
$\mathcal L$ is the language of ordered valued differential rings given by
$$\mathcal L\ :=\ \{0,\ 1,\  +,\  -,\ \cdot\ ,\ \der,\ \le,\ \preceq\}$$ 
where the unary operation symbol $\der$ names the derivation, and the binary
relation symbol $\preceq$ names the valuation 
divisibility on the field $\T$ given by 
$$f\preceq g\ \Longleftrightarrow\ |f|\le c|g| \text{ for some $c\in \R^{>0}$.}$$
For the $\T$-Conjecture, it doesn't really matter whether or not 
we include~$\le$ and~$\preceq$, since the ordering and the valuation 
divisibility are existentially definable in terms of the other
primitives: for $\preceq$, use that $\R$ is the field of constants 
for the derivation. (See also \cite[Section~14]{AvdD2}.)
A purely differential-algebraic formulation of the $\T$-Conjecture 
reads as follows: 

\medskip
\noindent
{\it For any differential polynomial $P$ over $\Q$
in $m+n$ variables there exists a differential polynomial $Q$ over $\Q$ in $m+p$ variables, for some~$p$ depending on~$P$, such
that for all $a\in \T^m$ the following equivalence holds:}
$$ \text{$P(a,b) =0$  for {\it some $b\in \T^n$}} \quad\Longleftrightarrow\quad 
   \text{$Q(a,c)\ne 0$ for {\it all $c\in \T^p$.}}$$
In logical terms: every existential formula in the language of
differential rings is equivalent 
in $\T$ to a universal formula in that language.  

Sections~\ref{sec:qf definability} and ~\ref{sec:obstructions} suggest that
a strong form of the $\T$-Conjecture (elimination of quantifiers in a 
reasonable language) will imply the following attractive
and intrinsic model-theoretic properties of $\T$:

\begin{itemize} 

\item  If $X\subseteq \T^n$ is definable, then $X\cap \R^n$ is semialgebraic.

\item  $\T$ is {\it asymptotically o-minimal\/}: for each definable $X\subseteq \T$ there is a~$b\in \T$ such that either $(b,+\infty)\subseteq X$ or
$(b,+\infty)\subseteq \T\setminus X$.

\item  $\T$ has NIP. (What this means is explained in Section~\ref{sec:qf definability}.)
\end{itemize}

\subsection{Positive evidence.} In Section~\ref{sec:qf definability} we establish
quantifier-free versions of the last three statements. 
Over the years, evidence for the $\T$-Conjecture has accumulated. 
For example, the value group of $\T$ equipped with a certain function 
induced by the derivation of $\T$ (the ``asymptotic couple'' of $\T$ as defined
in Section~\ref{sec:asymptotic couples} below) is model complete; see \cite{AvdD}.
The best evidence for the $\T$-Conjecture to date is 
the analysis by van der Hoeven in \cite{vdH}
of the set of zeros in $\T$ of any given differential polynomial in one variable over $\T$. Among other things, he proved the following Intermediate Value Theorem:

\begin{theorem}  
Given any differential polynomial $P(Y)\in \T\{Y\}$ and $f,h\in \T$ with $P(f)<0<P(h)$, 
there is $g\in \T$ with $f < g < h$ and  $P(g)=0$.
\end{theorem}

Here and later $K\{Y\}= K[Y, Y', Y'',\dots]$  is the ring of differential 
polynomials in the indeterminate $Y$ over a differential field $K$. The
proofs in \cite{vdH} make full use of the formal structure of $\T$ 
as an increasing union of Hahn fields. This makes it possible to apply
analytic techniques (fixed point theorems, compact-like operators, etc.) for solving algebraic differential equations; see also ~\cite{vdH:noeth}. 
Much of our work consists of recovering 
significant parts of \cite{vdH} under weak first-order assumptions 
on valued differential fields.

\subsection{The different flavors of $\T$.} In any precise inductive
construction of 
$\T$ we can impose various conditions on the so-called
{\em support\/} of a transseries, which is the ordered set of transmonomials
occurring in it with a non-zero coefficient. This leads to variants of the 
differential field $\T$; see for example the 
discussion in \cite{E} and \cite{vdH}. For the sake of definiteness, we take
here $\T$ to be the field $\RLE$ of logarithmic-exponential power series 
from \cite{DMM}, where supports are only required to be {\em anti-wellordered}; 
this is basically the weakest condition that can be imposed.

In \cite{vdH}, however, each transseries has a {\em gridbased\/} 
support contained in a finitely generated 
subgroup of the multiplicative group of transmonomials. This leads to a rather small differential 
subfield of our $\T$, but results such as the Intermediate Value Theorem 
in \cite{vdH} proved there for the gridbased version of $\T$ are known to
hold also for the $\T$ we consider here. Of course, we expect these variants 
of $\T$ all to be elementarily equivalent, and this is part of the
motivation for our $\T$-Conjecture. For this expectation to hold we would 
need also an explicit first-order axiomatization of the theory of $\T$, and
show that the various flavors of $\T$ all satisfy these axioms. 
At the end of Section~\ref{sec:new results} we conjecture such an 
axiomatization as part of a more explicit version of the $\T$-Conjecture.

Likewise, we expect \'{E}calle's differential field $\T^{\text{as}}$ of
accelero-summable transseries to be an elementary submodel of $\T$. (By the way,
$\T^{\text{as}}$ comes in similar variants as $\T$ itself.) Also 
$\T^{\operatorname{da}}$, whose elements are the differentially algebraic 
transseries, is a natural 
candidate for an elementary submodel of $\T$.

\subsection{Linear differential operators over $\T$.}  
The Intermediate Value Property for differential polynomials over $\T$ resembles
the behavior of ordinary one-variable polynomials over $\R$. 
There is another analogy in~\cite{vdH} between $\T$ and~$\R$ 
which is much easier to establish: factoring
linear differential operators over $\T$ is similar to factoring
one-variable polynomials over $\R$. By a linear differential operator over 
$\T$ we mean an operator
$A=a_0 + a_1\der + \dots + a_n\der^n$ on $\T\ $  
($\der=\text{the derivation}$,  all $a_i\in \T$);   it
 defines the same function on $\T$ as the differential polynomial 
$a_0 Y + a_1Y' + \dots + a_nY^{(n)}$.
The linear differential operators over $\T$ form a {\em non-commutative\/} 
ring $\T[\der]$ under composition.

\begin{theorem}  Every linear differential operator over $\T$  of positive 
order is surjective as a map $\T \to \T$, and
 is a product \textup{(}composition\textup{)}  of operators 
$a+b\der$ of order $1$ in $\T[\imag][\der]$. Every 
such operator is a product of order $1$ and order $2$ operators in 
$\T[\partial]$.  
\end{theorem}

Thus coming to grips with linear differential operators over $\T$ 
reduces to some extent to understanding those of order $1$ and order $2$. 
Studying operators of order $1$ is largely a matter of solving equations 
$y'=a$ and $z^\dagger =b$. Modulo solving such equations, 
order $2$ operators can be reduced to those
of the form $4\der^2+f$, where the next theorem is relevant.

\begin{theorem}\label{thm:2ndorder} Let $f\in \T$. Then the following are equivalent: 
\begin{enumerate}
\item[(1)] the equation $4y''+fy\ =\ 0$ has a non-zero solution in $\T$;
\item[(2)] $f\ <\  \frac{1}{(\ell_0)^2}+\frac{1}{(\ell_0\ell_1)^2}+
\frac{1}{(\ell_0\ell_1\ell_2)^2} +\cdots+
\frac{1}{(\ell_0\ell_1\cdots\ell_n)^2}$
for some $n$;
\item[(3)] $f\neq 2(u^\dagger{}^\dagger)'-(u^\dagger{}^\dagger)^2 +(u^\dagger)^2$
for all $u>\R$ in $\T$.
\end{enumerate}
\end{theorem}

The equivalence of (1) and (2) is analogous to a theorem of Boshernitzan~\cite{Boshernitzan} and Rosenlicht~\cite{Rosenlicht95} in the realm of Hardy fields. (See the remarks following Theorem~1.12 in \cite{AvdD3} for a correction of \cite{Rosenlicht95}.) The equivalence of (1) and (3) has been known to us since
2002. Its model-theoretic significance is that
the existential condition (1) on $f$ is equivalent to a universal condition on $f$, 
namely (3), in accordance with the $\T$-Conjecture.

We note here that for a non-constant element $u$ of a differential field,
\begin{align*} 2(u^\dagger{}^\dagger)'-(u^\dagger{}^\dagger)^2+(u^\dagger)^2\  &=\ 2S(u),\  \text{ where}\\
S(u)\ :=\ (u^{\prime\dagger})' - \frac{1}{2} (u^{\prime\dagger})^2\ &=\
\frac{u'''}{u'} - \frac{3}{2}\left(\frac{u''}{u'}\right)^2
\end{align*}
is known as the Schwarzian derivative of $u$, which plays a role in the analytic theory of linear differential equations; see \cite[Chapter~10]{Hille}.

\section{$H$-Fields}\label{sec:H-fields}

\noindent
Abraham Robinson taught us to think about model completeness and quantifier
elimination in an abstract algebraic way. This approach as refined by Shoenfield
and Blum suggests that the $\T$-Conjecture follows from an
adequate extension theory for those ordered differential fields 
that share certain basic (universal) properties with $\T$. This involves
a critical choice of the
``right'' class of ordered differential fields. Our choice: 
$H$-fields\footnote{The prefix~$H$ honors the pioneers Hahn, 
Hardy, and Hausdorff. Arguably, Borel's work~\cite{Borel} 
in this vein is even more significant, but his name doesn't start with H. 
One could go still further back, to du Bois-Reymond's paper~\cite{duBoisReymond}, a source of inspiration for Hardy~\cite{hardy}.} 
as defined below. 
 Then the challenge becomes to show that 
the  ``existentially closed'' $H$-fields are exactly the
$H$-fields that share certain deeper first-order properties with~$\T$. 
If we can achieve this, then $\T$ will be model complete. 

In practice this
often amounts to the following: come up with the ``right'' extra primitives
(these should be existentially as well as universally definable in~$\T$); guess
the ``right'' axioms characterizing existentially closed $H$-fields;
and prove suitable embedding theorems for $H$-fields enriched with
these primitives. 
If this works, one has a proof of  a strong form of the
$\T$-Conjecture, namely an elimination of
quantifiers in the language $\mathcal L$ augmented by symbols for the extra primitives.
Such an approach to understanding definability in a given mathematical 
structure often yields further payoffs, for example, 
a useful dimension theory for definable sets.

\medskip\noindent
Let $K$ be an ordered differential field, and put\footnote{The notations $\mathcal{O}$ and $\smallo$ are reminders of Landau's big O and small o.} 
\begin{align*}  C&=\{a\in K: a'=0\} && \text{(constant field of $K$)}\\
\mathcal{O} & =  \{a\in K: \abs{a}\le c \text{ for some $c\in C^{>0}$}\} &&    \text{(convex hull of $C$ in $K$)}\\
              \smallo  &= \{a\in K: \abs{a} < c \text{ for all $c\in C^{>0}$}\}  && \text{(maximal ideal of $\mathcal{O}$).}
\end{align*} 
We call $K$  an {\bf $H$-field} if the following conditions are satisfied: 
\begin{enumerate} 
\item[(H1)]   $\mathcal{O}\  =\  C+\smallo$,
\item[(H2)]   $a>C\ \Longrightarrow\ a'>0$.
\end{enumerate}
Examples of $H$-fields include any Hardy field containing $\R$, such as $\R(x, \ex^x)$; the ordered differential field  $ \R \(( x^{-1}  \)) $
of Laurent series; and~$\T$. All these satisfy 
an extra axiom: 
\begin{enumerate} 
\item[(H3)]   $a\in \smallo\ \Longrightarrow\ a' \in \smallo$,
\end{enumerate}
which is also expressed by saying that the derivation is 
{\bf small.}

\medskip\noindent
An $H$-field $K$ comes with a definable (Krull) valuation 
$v$ whose valuation ring is the convex hull $\mathcal{O}$ of $C$.
It will be useful to fix some 
notation for any valued differential field $K$, not necessarily an
$H$-field: $C$ is the constant field, $\mathcal{O}$ is the valuation ring,
$\smallo$ is the maximal ideal of $\mathcal{O}$, and 
$v\colon K^\times \to \Gamma$ with $\Gamma=v(K^{\times})$ is the valuation. If we need to indicate 
the dependence on $K$ we use subscripts, so 
$C=C_K$, $\mathcal{O}=\mathcal{O}_K$, and so on. The 
valuation divisibility on~$K$ corresponding to its valuation is 
the binary relation $\preceq$ on $K$ given by
$$ f\preceq g\ \Longleftrightarrow\ vf \ge vg.$$
Note that if $K$ is an $H$-field, then for all $f,g\in K$,  
$$f\preceq g \Longleftrightarrow\ \text{$|f|\le c|g|$ for some $c\in C^{>0}$.}$$
We also write $g\succeq f$ instead of $f\preceq g$, and we define 
$$f\asymp g\ \Longleftrightarrow\ \text{$f\preceq g$ and $g\preceq f$,} 
\qquad f\sim g\  \Longleftrightarrow\ f-g\prec f.$$
Further, we introduce the binary relations $\prec$ and $\succ$ on $K$:  
$$ f\prec g\ \Longleftrightarrow\ f\preceq g \text{ and }f\not\asymp g\Longleftrightarrow\ vf > vg, \qquad f\succ g\ \Longleftrightarrow\ g\prec f.$$
If $K$ is an $H$-field, then for $f,g\in K$ this means: 
$$f\prec g\ \Longleftrightarrow\ \text{$|f|< c|g|$ for all $c\in C^{>0}$.}$$
Rosenlicht gave a nice valuation-theoretic formulation of 
l'H\^opital's rule: if~$K$ is a Hardy field, then  
\begin{equation}\tag{$\ast$}\label{eq:lH}
\text{ for all $f, g\in K$ with $f,g\prec 1$:\  $\ f\prec g\ \Longleftrightarrow\ f'\prec g'$.}
\end{equation}
This rule \eqref{eq:lH} goes through for $H$-fields. 
The ordering of an $H$-field determines its valuation, but plays 
otherwise a 
secondary role. Moreover, it is often useful to pass to algebraic
closures like $\T[\imag]$, with the valuation extending uniquely, 
still obeying (H1) and \eqref{eq:lH}, but without
ordering. Thus much of our work is in the 
setting of {\bf asymptotic differential fields}: these are
the valued differential 
fields
satisfying~\eqref{eq:lH}. We use ``asymptotic field'' as abbreviation for ``asymptotic differential field''. Section~\ref{sec:new results} will show the benefits of
coarsening the valuation of an $H$-field; the resulting object might not be an 
$H$-field anymore, but remains an asymptotic field.
It is a useful and
non-trivial fact that any algebraic extension of an asymptotic  
field is also an asymptotic field.

\medskip\noindent
An $H$-field $K$ is {\bf existentially closed}  
if every finite system of algebraic differential equations over $K$ 
in several unknowns with a solution in an $H$-field extension of $K$ 
has a solution in $K$. Including in these systems also differential 
inequalities (using~$\le $ and~$<$) and asymptotic conditions (involving~$\preceq$ and~$\prec$) makes no
difference. (See \cite[Section~14]{AvdD2}.) A more detailed version of the $\T$-Conjecture now says:

\begin{refinedTconjecture}
$\T$ is an existentially closed $H$-field, and there exists a set
$\Sigma$ of $\mathcal L$-sentences such that the existentially closed
$H$-fields with small derivation are exactly the $H$-fields
satisfying $\Sigma$. $($In more model-theoretic jargon: 
the theory of $H$-fields with 
small derivation has a model companion, and $\T$ is a model of
this model companion.$)$  
\end{refinedTconjecture}

A comment on axiom (H1) for $H$-fields: it expresses that the
{\em constant\/} field for the derivation is also in a natural way the 
{\em residue\/}
field for the valuation. However, (H1) cannot be expressed by a 
universal sentence in the language $\mathcal L$ of
ordered valued differential rings. We define a pre-$H$-field to be an
ordered valued differential subfield of an $H$-field. There are
pre-$H$-fields that are not $H$-fields, and the valuation of a pre-$H$-field
is not always determined by its ordering, as is the case in $H$-fields. 
Fortunately, any pre-$H$-field $K$ has an $H$-field extension $H(K)$, its
{\bf $H$-field closure},
that embeds uniquely over $K$ into any $H$-field extension of $K$; 
see \cite{AvdD1}. (Here and below, ``extension'' and ``embedding'' are
meant in the sense of $\mathcal L$-structures.)    

\medskip
\noindent
Figure~\ref{fig:ordered differential fields} indicates the inclusions among
the various classes of ordered valued differential fields defined in this section, except that asymptotic fields are not necessarily ordered. The right half
represents the case of {\em small derivation}.

\begin{figure}
\begin{picture}(300,200)

\put(0,0){\line(0,1){200}}
\put(0,0){\line(1,0){300}}

\put(20,20){\line(0,1){160}}
\put(20,20){\line(1,0){260}}

\put(40,40){\line(0,1){120}}
\put(40,40){\line(1,0){220}}

\put(60,60){\line(0,1){80}}
\put(60,60){\line(1,0){180}}

\put(0,200){\line(1,0){300}}
\put(20,180){\line(1,0){260}}
\put(40,160){\line(1,0){220}}
\put(60,140){\line(1,0){180}}

\put(300,0){\line(0,1){200}}
\put(280,20){\line(0,1){160}}
\put(260,40){\line(0,1){120}}
\put(240,60){\line(0,1){80}}

\put(80,80){\line(0,1){40}}
\put(80,80){\line(1,0){140}}
\put(80,120){\line(1,0){140}}
\put(220,80){\line(0,1){40}}

\multiput(150,95)(0,4){26}{\line(0,1){2.5}}
\multiput(150,0)(0,4){16}{\line(0,1){2.5}}
\multiput(150,75)(0,4){2}{\line(0,1){2.5}}


\put(10,5){asymptotic fields}
\put(30,25){pre-$H$-fields}
\put(50,45){$H$-fields}
\put(70,65){Liouville closed $H$-fields}
\put(90,85){existentially closed $H$-fields}
\put(170,187){with small derivation}

\end{picture}
\caption{}
\label{fig:ordered differential fields}
\end{figure}

\subsection{Liouville closed $H$-fields.} The real closure of an $H$-field is 
again an $H$-field; see \cite{AvdD1}. 
Going beyond algebraic adjunctions, we consider adjoining solutions to
first-order linear differential equations $y'+ay=b$. 

\medskip\noindent
Call an $H$-field $K$ {\bf Liouville closed} if it is real closed and for all $a,b \in K$ there are $y,z\in K$ such that $y'=a$  and
$z\ne 0$, $z^\dagger =b$; equivalently, $K$ is real closed, and
any equation $y'+ay=b$ with $a,b\in K$ has a {\em non-zero}\footnote{This non-zero requirement was inadvertently dropped on p.~580 of~\cite{AvdD1}.} solution $y\in K$.  
For example, $\T$ is Liouville closed. Each existentially closed $H$-field is
Liouville closed as a consequence of the next theorem. 
A {\bf Liouville closure} of an $H$-field $K$ is a minimal Liouville closed $H$-field extension of $K$.
We can now state the main result from \cite{AvdD1}:

\begin{theorem} Let $K$ be an $H$-field. Then $K$ has exactly one 
Liouville closure, or exactly two 
Liouville closures \textup{(}up to isomorphism over $K$\textup{)}.
\end{theorem}

Whether $K$ has one or two Liouville closures is related to a 
trichotomy in the class of $H$-fields which
pervades our work. In fact, it is a trichotomy that can be
detected on the level of the value group; see below.

\subsection{Trichotomy for $H$-fields.} \label{sec:trichotomy}

Let $K$ be an asymptotic field with valuation~$v$ and value group
$\Gamma= v(K^\times)$. We set 
$$\Gamma^{\ne}:= \Gamma\setminus \{0\}, \quad \Gamma^{<}:= \{\gamma\in \Gamma:\ \gamma<0\}, \quad \Gamma^{>}:=\{\gamma\in \Gamma:\ \gamma>0\}.$$ It follows from 
the l'H\^{o}pital-Rosenlicht rule \eqref{eq:lH} that
the derivation and the logarithmic derivative of $K$ induce functions
on $\Gamma^{\ne}$:  
\begin{align*}v(a)= \gamma   &\mapsto v(a')=\gamma'\   :\  \Gamma^{\ne} \to   \Gamma, \\
       v(a)=\gamma &\mapsto v(a^\dagger)=\gamma^\dagger:= \gamma'-\gamma \   :\  \Gamma^{\ne} \to   \Gamma,
\end{align*}  
where $a\in K^\times$, $v(a)\ne 0$. 
The function $\gamma\mapsto\gamma'\colon \Gamma^{\neq}\to\Gamma$ is strictly increasing and the function $\gamma\mapsto\gamma^\dagger\colon\Gamma^{\ne}\to\Gamma$ is symmetric:
$(-\gamma)^\dagger=\gamma^\dagger$ for all $\gamma\in\Gamma^{\neq}$. If $K$ is an $H$-field, then $\gamma\mapsto \gamma^{\dagger}: \Gamma^{>}\to \Gamma$ is decreasing.
Figure~\ref{psi-pic} shows the qualitative behavior 
of the functions $\gamma\mapsto\gamma'$ and
$\gamma\mapsto\gamma^\dagger$ in the case of an $H$-field. Some features are a little hard to indicate in 
such a picture, for example the fact that $\gamma^{\dagger}$ is 
constant on each archimedean class of $\Gamma^{\ne}$.  

\begin{figure}
\begin{center}
\setlength{\unitlength}{0.240000pt}
\ifx\plotpoint\undefined\newsavebox{\plotpoint}\fi
\sbox{\plotpoint}{\rule[-0.175pt]{0.350pt}{0.350pt}}%
\begin{picture}(1400,800)(150,100)
\put(264,472){\rule[-0.175pt]{282.335pt}{0.350pt}}
\put(850,158){\rule[-0.175pt]{0.350pt}{151.526pt}}
\put(749,787){\makebox(0,0)[l]{\shortstack{$\Gamma\,\uparrow$}}}
\put(1400,518){\makebox(0,0){$\rightarrow\ \Gamma$}}
\put(850,556){\makebox(0,0){$\circ$}}
\put(1163,822){\makebox(0,0)[l]{$\gamma'$}}
\put(1241,403){\makebox(0,0)[l]{$\gamma^\dagger$}}
\put(264,445){\usebox{\plotpoint}}
\put(264,445){\rule[-0.175pt]{31.317pt}{0.350pt}}
\put(394,446){\rule[-0.175pt]{8.672pt}{0.350pt}}
\put(430,447){\rule[-0.175pt]{5.541pt}{0.350pt}}
\put(453,448){\rule[-0.175pt]{2.891pt}{0.350pt}}
\put(465,449){\rule[-0.175pt]{2.891pt}{0.350pt}}
\put(477,450){\rule[-0.175pt]{2.891pt}{0.350pt}}
\put(489,451){\rule[-0.175pt]{2.891pt}{0.350pt}}
\put(501,452){\rule[-0.175pt]{1.445pt}{0.350pt}}
\put(507,453){\rule[-0.175pt]{1.445pt}{0.350pt}}
\put(513,454){\rule[-0.175pt]{1.325pt}{0.350pt}}
\put(518,455){\rule[-0.175pt]{1.325pt}{0.350pt}}
\put(524,456){\rule[-0.175pt]{1.445pt}{0.350pt}}
\put(530,457){\rule[-0.175pt]{1.445pt}{0.350pt}}
\put(536,458){\rule[-0.175pt]{1.445pt}{0.350pt}}
\put(542,459){\rule[-0.175pt]{1.445pt}{0.350pt}}
\put(548,460){\rule[-0.175pt]{1.445pt}{0.350pt}}
\put(554,461){\rule[-0.175pt]{1.445pt}{0.350pt}}
\put(560,462){\rule[-0.175pt]{0.964pt}{0.350pt}}
\put(564,463){\rule[-0.175pt]{0.964pt}{0.350pt}}
\put(568,464){\rule[-0.175pt]{0.964pt}{0.350pt}}
\put(572,465){\rule[-0.175pt]{0.964pt}{0.350pt}}
\put(576,466){\rule[-0.175pt]{0.964pt}{0.350pt}}
\put(580,467){\rule[-0.175pt]{0.964pt}{0.350pt}}
\put(584,468){\rule[-0.175pt]{0.662pt}{0.350pt}}
\put(586,469){\rule[-0.175pt]{0.662pt}{0.350pt}}
\put(589,470){\rule[-0.175pt]{0.662pt}{0.350pt}}
\put(592,471){\rule[-0.175pt]{0.662pt}{0.350pt}}
\put(595,472){\rule[-0.175pt]{0.964pt}{0.350pt}}
\put(599,473){\rule[-0.175pt]{0.964pt}{0.350pt}}
\put(603,474){\rule[-0.175pt]{0.964pt}{0.350pt}}
\put(607,475){\rule[-0.175pt]{0.723pt}{0.350pt}}
\put(610,476){\rule[-0.175pt]{0.723pt}{0.350pt}}
\put(613,477){\rule[-0.175pt]{0.723pt}{0.350pt}}
\put(616,478){\rule[-0.175pt]{0.723pt}{0.350pt}}
\put(619,479){\rule[-0.175pt]{0.578pt}{0.350pt}}
\put(621,480){\rule[-0.175pt]{0.578pt}{0.350pt}}
\put(623,481){\rule[-0.175pt]{0.578pt}{0.350pt}}
\put(626,482){\rule[-0.175pt]{0.578pt}{0.350pt}}
\put(628,483){\rule[-0.175pt]{0.578pt}{0.350pt}}
\put(631,484){\rule[-0.175pt]{0.723pt}{0.350pt}}
\put(634,485){\rule[-0.175pt]{0.723pt}{0.350pt}}
\put(637,486){\rule[-0.175pt]{0.723pt}{0.350pt}}
\put(640,487){\rule[-0.175pt]{0.723pt}{0.350pt}}
\put(643,488){\rule[-0.175pt]{0.578pt}{0.350pt}}
\put(645,489){\rule[-0.175pt]{0.578pt}{0.350pt}}
\put(647,490){\rule[-0.175pt]{0.578pt}{0.350pt}}
\put(650,491){\rule[-0.175pt]{0.578pt}{0.350pt}}
\put(652,492){\rule[-0.175pt]{0.578pt}{0.350pt}}
\put(655,493){\rule[-0.175pt]{0.578pt}{0.350pt}}
\put(657,494){\rule[-0.175pt]{0.578pt}{0.350pt}}
\put(659,495){\rule[-0.175pt]{0.578pt}{0.350pt}}
\put(662,496){\rule[-0.175pt]{0.578pt}{0.350pt}}
\put(664,497){\rule[-0.175pt]{0.578pt}{0.350pt}}
\put(667,498){\rule[-0.175pt]{0.530pt}{0.350pt}}
\put(669,499){\rule[-0.175pt]{0.530pt}{0.350pt}}
\put(671,500){\rule[-0.175pt]{0.530pt}{0.350pt}}
\put(673,501){\rule[-0.175pt]{0.530pt}{0.350pt}}
\put(675,502){\rule[-0.175pt]{0.530pt}{0.350pt}}
\put(678,503){\rule[-0.175pt]{0.482pt}{0.350pt}}
\put(680,504){\rule[-0.175pt]{0.482pt}{0.350pt}}
\put(682,505){\rule[-0.175pt]{0.482pt}{0.350pt}}
\put(684,506){\rule[-0.175pt]{0.482pt}{0.350pt}}
\put(686,507){\rule[-0.175pt]{0.482pt}{0.350pt}}
\put(688,508){\rule[-0.175pt]{0.482pt}{0.350pt}}
\put(690,509){\rule[-0.175pt]{0.578pt}{0.350pt}}
\put(692,510){\rule[-0.175pt]{0.578pt}{0.350pt}}
\put(694,511){\rule[-0.175pt]{0.578pt}{0.350pt}}
\put(697,512){\rule[-0.175pt]{0.578pt}{0.350pt}}
\put(699,513){\rule[-0.175pt]{0.578pt}{0.350pt}}
\put(702,514){\rule[-0.175pt]{0.578pt}{0.350pt}}
\put(704,515){\rule[-0.175pt]{0.578pt}{0.350pt}}
\put(706,516){\rule[-0.175pt]{0.578pt}{0.350pt}}
\put(709,517){\rule[-0.175pt]{0.578pt}{0.350pt}}
\put(711,518){\rule[-0.175pt]{0.578pt}{0.350pt}}
\put(714,519){\rule[-0.175pt]{0.578pt}{0.350pt}}
\put(716,520){\rule[-0.175pt]{0.578pt}{0.350pt}}
\put(718,521){\rule[-0.175pt]{0.578pt}{0.350pt}}
\put(721,522){\rule[-0.175pt]{0.578pt}{0.350pt}}
\put(723,523){\rule[-0.175pt]{0.578pt}{0.350pt}}
\put(726,524){\rule[-0.175pt]{0.578pt}{0.350pt}}
\put(728,525){\rule[-0.175pt]{0.578pt}{0.350pt}}
\put(730,526){\rule[-0.175pt]{0.578pt}{0.350pt}}
\put(733,527){\rule[-0.175pt]{0.578pt}{0.350pt}}
\put(735,528){\rule[-0.175pt]{0.578pt}{0.350pt}}
\put(738,529){\rule[-0.175pt]{0.530pt}{0.350pt}}
\put(740,530){\rule[-0.175pt]{0.530pt}{0.350pt}}
\put(742,531){\rule[-0.175pt]{0.530pt}{0.350pt}}
\put(744,532){\rule[-0.175pt]{0.530pt}{0.350pt}}
\put(746,533){\rule[-0.175pt]{0.530pt}{0.350pt}}
\put(749,534){\rule[-0.175pt]{0.578pt}{0.350pt}}
\put(751,535){\rule[-0.175pt]{0.578pt}{0.350pt}}
\put(753,536){\rule[-0.175pt]{0.578pt}{0.350pt}}
\put(756,537){\rule[-0.175pt]{0.578pt}{0.350pt}}
\put(758,538){\rule[-0.175pt]{0.578pt}{0.350pt}}
\put(761,539){\rule[-0.175pt]{0.723pt}{0.350pt}}
\put(764,540){\rule[-0.175pt]{0.723pt}{0.350pt}}
\put(767,541){\rule[-0.175pt]{0.723pt}{0.350pt}}
\put(770,542){\rule[-0.175pt]{0.723pt}{0.350pt}}
\put(773,543){\rule[-0.175pt]{0.964pt}{0.350pt}}
\put(777,544){\rule[-0.175pt]{0.964pt}{0.350pt}}
\put(781,545){\rule[-0.175pt]{0.964pt}{0.350pt}}
\put(785,546){\rule[-0.175pt]{0.723pt}{0.350pt}}
\put(788,547){\rule[-0.175pt]{0.723pt}{0.350pt}}
\put(791,548){\rule[-0.175pt]{0.723pt}{0.350pt}}
\put(794,549){\rule[-0.175pt]{0.723pt}{0.350pt}}
\put(797,550){\rule[-0.175pt]{1.445pt}{0.350pt}}
\put(803,551){\rule[-0.175pt]{1.445pt}{0.350pt}}
\put(809,552){\rule[-0.175pt]{1.325pt}{0.350pt}}
\put(814,553){\rule[-0.175pt]{1.325pt}{0.350pt}}
\put(820,554){\rule[-0.175pt]{1.445pt}{0.350pt}}
\put(826,555){\rule[-0.175pt]{1.445pt}{0.350pt}}
\put(832,556){\rule[-0.175pt]{2.500pt}{0.350pt}}
\put(856,556){\rule[-0.175pt]{4.300pt}{0.350pt}}
\put(874,555){\rule[-0.175pt]{1.445pt}{0.350pt}}
\put(880,554){\rule[-0.175pt]{1.325pt}{0.350pt}}
\put(885,553){\rule[-0.175pt]{1.325pt}{0.350pt}}
\put(891,552){\rule[-0.175pt]{1.445pt}{0.350pt}}
\put(897,551){\rule[-0.175pt]{1.445pt}{0.350pt}}
\put(903,550){\rule[-0.175pt]{0.723pt}{0.350pt}}
\put(906,549){\rule[-0.175pt]{0.723pt}{0.350pt}}
\put(909,548){\rule[-0.175pt]{0.723pt}{0.350pt}}
\put(912,547){\rule[-0.175pt]{0.723pt}{0.350pt}}
\put(915,546){\rule[-0.175pt]{0.964pt}{0.350pt}}
\put(919,545){\rule[-0.175pt]{0.964pt}{0.350pt}}
\put(923,544){\rule[-0.175pt]{0.964pt}{0.350pt}}
\put(927,543){\rule[-0.175pt]{0.723pt}{0.350pt}}
\put(930,542){\rule[-0.175pt]{0.723pt}{0.350pt}}
\put(933,541){\rule[-0.175pt]{0.723pt}{0.350pt}}
\put(936,540){\rule[-0.175pt]{0.723pt}{0.350pt}}
\put(939,539){\rule[-0.175pt]{0.578pt}{0.350pt}}
\put(941,538){\rule[-0.175pt]{0.578pt}{0.350pt}}
\put(943,537){\rule[-0.175pt]{0.578pt}{0.350pt}}
\put(946,536){\rule[-0.175pt]{0.578pt}{0.350pt}}
\put(948,535){\rule[-0.175pt]{0.578pt}{0.350pt}}
\put(951,534){\rule[-0.175pt]{0.530pt}{0.350pt}}
\put(953,533){\rule[-0.175pt]{0.530pt}{0.350pt}}
\put(955,532){\rule[-0.175pt]{0.530pt}{0.350pt}}
\put(957,531){\rule[-0.175pt]{0.530pt}{0.350pt}}
\put(959,530){\rule[-0.175pt]{0.530pt}{0.350pt}}
\put(962,529){\rule[-0.175pt]{0.578pt}{0.350pt}}
\put(964,528){\rule[-0.175pt]{0.578pt}{0.350pt}}
\put(966,527){\rule[-0.175pt]{0.578pt}{0.350pt}}
\put(969,526){\rule[-0.175pt]{0.578pt}{0.350pt}}
\put(971,525){\rule[-0.175pt]{0.578pt}{0.350pt}}
\put(974,524){\rule[-0.175pt]{0.578pt}{0.350pt}}
\put(976,523){\rule[-0.175pt]{0.578pt}{0.350pt}}
\put(978,522){\rule[-0.175pt]{0.578pt}{0.350pt}}
\put(981,521){\rule[-0.175pt]{0.578pt}{0.350pt}}
\put(983,520){\rule[-0.175pt]{0.578pt}{0.350pt}}
\put(986,519){\rule[-0.175pt]{0.578pt}{0.350pt}}
\put(988,518){\rule[-0.175pt]{0.578pt}{0.350pt}}
\put(990,517){\rule[-0.175pt]{0.578pt}{0.350pt}}
\put(993,516){\rule[-0.175pt]{0.578pt}{0.350pt}}
\put(995,515){\rule[-0.175pt]{0.578pt}{0.350pt}}
\put(998,514){\rule[-0.175pt]{0.578pt}{0.350pt}}
\put(1000,513){\rule[-0.175pt]{0.578pt}{0.350pt}}
\put(1002,512){\rule[-0.175pt]{0.578pt}{0.350pt}}
\put(1005,511){\rule[-0.175pt]{0.578pt}{0.350pt}}
\put(1007,510){\rule[-0.175pt]{0.578pt}{0.350pt}}
\put(1010,509){\rule[-0.175pt]{0.482pt}{0.350pt}}
\put(1012,508){\rule[-0.175pt]{0.482pt}{0.350pt}}
\put(1014,507){\rule[-0.175pt]{0.482pt}{0.350pt}}
\put(1016,506){\rule[-0.175pt]{0.482pt}{0.350pt}}
\put(1018,505){\rule[-0.175pt]{0.482pt}{0.350pt}}
\put(1020,504){\rule[-0.175pt]{0.482pt}{0.350pt}}
\put(1022,503){\rule[-0.175pt]{0.530pt}{0.350pt}}
\put(1024,502){\rule[-0.175pt]{0.530pt}{0.350pt}}
\put(1026,501){\rule[-0.175pt]{0.530pt}{0.350pt}}
\put(1028,500){\rule[-0.175pt]{0.530pt}{0.350pt}}
\put(1030,499){\rule[-0.175pt]{0.530pt}{0.350pt}}
\put(1032,498){\rule[-0.175pt]{0.578pt}{0.350pt}}
\put(1035,497){\rule[-0.175pt]{0.578pt}{0.350pt}}
\put(1037,496){\rule[-0.175pt]{0.578pt}{0.350pt}}
\put(1040,495){\rule[-0.175pt]{0.578pt}{0.350pt}}
\put(1042,494){\rule[-0.175pt]{0.578pt}{0.350pt}}
\put(1045,493){\rule[-0.175pt]{0.578pt}{0.350pt}}
\put(1047,492){\rule[-0.175pt]{0.578pt}{0.350pt}}
\put(1049,491){\rule[-0.175pt]{0.578pt}{0.350pt}}
\put(1052,490){\rule[-0.175pt]{0.578pt}{0.350pt}}
\put(1054,489){\rule[-0.175pt]{0.578pt}{0.350pt}}
\put(1057,488){\rule[-0.175pt]{0.723pt}{0.350pt}}
\put(1060,487){\rule[-0.175pt]{0.723pt}{0.350pt}}
\put(1063,486){\rule[-0.175pt]{0.723pt}{0.350pt}}
\put(1066,485){\rule[-0.175pt]{0.723pt}{0.350pt}}
\put(1069,484){\rule[-0.175pt]{0.578pt}{0.350pt}}
\put(1071,483){\rule[-0.175pt]{0.578pt}{0.350pt}}
\put(1073,482){\rule[-0.175pt]{0.578pt}{0.350pt}}
\put(1076,481){\rule[-0.175pt]{0.578pt}{0.350pt}}
\put(1078,480){\rule[-0.175pt]{0.578pt}{0.350pt}}
\put(1081,479){\rule[-0.175pt]{0.723pt}{0.350pt}}
\put(1084,478){\rule[-0.175pt]{0.723pt}{0.350pt}}
\put(1087,477){\rule[-0.175pt]{0.723pt}{0.350pt}}
\put(1090,476){\rule[-0.175pt]{0.723pt}{0.350pt}}
\put(1093,475){\rule[-0.175pt]{0.964pt}{0.350pt}}
\put(1097,474){\rule[-0.175pt]{0.964pt}{0.350pt}}
\put(1101,473){\rule[-0.175pt]{0.964pt}{0.350pt}}
\put(1105,472){\rule[-0.175pt]{0.662pt}{0.350pt}}
\put(1107,471){\rule[-0.175pt]{0.662pt}{0.350pt}}
\put(1110,470){\rule[-0.175pt]{0.662pt}{0.350pt}}
\put(1113,469){\rule[-0.175pt]{0.662pt}{0.350pt}}
\put(1116,468){\rule[-0.175pt]{0.964pt}{0.350pt}}
\put(1120,467){\rule[-0.175pt]{0.964pt}{0.350pt}}
\put(1124,466){\rule[-0.175pt]{0.964pt}{0.350pt}}
\put(1128,465){\rule[-0.175pt]{0.964pt}{0.350pt}}
\put(1132,464){\rule[-0.175pt]{0.964pt}{0.350pt}}
\put(1136,463){\rule[-0.175pt]{0.964pt}{0.350pt}}
\put(1140,462){\rule[-0.175pt]{1.445pt}{0.350pt}}
\put(1146,461){\rule[-0.175pt]{1.445pt}{0.350pt}}
\put(1152,460){\rule[-0.175pt]{1.445pt}{0.350pt}}
\put(1158,459){\rule[-0.175pt]{1.445pt}{0.350pt}}
\put(1164,458){\rule[-0.175pt]{1.445pt}{0.350pt}}
\put(1170,457){\rule[-0.175pt]{1.445pt}{0.350pt}}
\put(1176,456){\rule[-0.175pt]{1.325pt}{0.350pt}}
\put(1181,455){\rule[-0.175pt]{1.325pt}{0.350pt}}
\put(1187,454){\rule[-0.175pt]{1.445pt}{0.350pt}}
\put(1193,453){\rule[-0.175pt]{1.445pt}{0.350pt}}
\put(1199,452){\rule[-0.175pt]{2.891pt}{0.350pt}}
\put(1211,451){\rule[-0.175pt]{2.891pt}{0.350pt}}
\put(1223,450){\rule[-0.175pt]{2.891pt}{0.350pt}}
\put(1235,449){\rule[-0.175pt]{2.891pt}{0.350pt}}
\put(1247,448){\rule[-0.175pt]{2.650pt}{0.350pt}}
\put(1258,447){\rule[-0.175pt]{5.782pt}{0.350pt}}
\put(1282,446){\rule[-0.175pt]{8.672pt}{0.350pt}}
\put(1318,445){\rule[-0.175pt]{28.426pt}{0.350pt}}
\put(518,158){\usebox{\plotpoint}}
\put(519,159){\usebox{\plotpoint}}
\put(520,160){\usebox{\plotpoint}}
\put(521,161){\usebox{\plotpoint}}
\put(522,162){\usebox{\plotpoint}}
\put(523,163){\usebox{\plotpoint}}
\put(524,164){\usebox{\plotpoint}}
\put(525,165){\usebox{\plotpoint}}
\put(526,166){\usebox{\plotpoint}}
\put(527,167){\usebox{\plotpoint}}
\put(528,168){\usebox{\plotpoint}}
\put(529,169){\usebox{\plotpoint}}
\put(530,170){\usebox{\plotpoint}}
\put(531,171){\usebox{\plotpoint}}
\put(532,172){\usebox{\plotpoint}}
\put(533,173){\usebox{\plotpoint}}
\put(534,174){\usebox{\plotpoint}}
\put(535,175){\usebox{\plotpoint}}
\put(536,176){\usebox{\plotpoint}}
\put(537,178){\usebox{\plotpoint}}
\put(538,179){\usebox{\plotpoint}}
\put(539,180){\usebox{\plotpoint}}
\put(540,181){\usebox{\plotpoint}}
\put(541,182){\usebox{\plotpoint}}
\put(542,183){\usebox{\plotpoint}}
\put(543,184){\usebox{\plotpoint}}
\put(544,185){\usebox{\plotpoint}}
\put(545,186){\usebox{\plotpoint}}
\put(546,187){\usebox{\plotpoint}}
\put(547,188){\usebox{\plotpoint}}
\put(548,189){\usebox{\plotpoint}}
\put(549,191){\usebox{\plotpoint}}
\put(550,192){\usebox{\plotpoint}}
\put(551,193){\usebox{\plotpoint}}
\put(552,194){\usebox{\plotpoint}}
\put(553,195){\usebox{\plotpoint}}
\put(554,196){\usebox{\plotpoint}}
\put(555,197){\usebox{\plotpoint}}
\put(556,198){\usebox{\plotpoint}}
\put(557,199){\usebox{\plotpoint}}
\put(558,200){\usebox{\plotpoint}}
\put(559,201){\usebox{\plotpoint}}
\put(560,202){\usebox{\plotpoint}}
\put(561,204){\usebox{\plotpoint}}
\put(562,205){\usebox{\plotpoint}}
\put(563,206){\usebox{\plotpoint}}
\put(564,207){\usebox{\plotpoint}}
\put(565,208){\usebox{\plotpoint}}
\put(566,209){\usebox{\plotpoint}}
\put(567,210){\usebox{\plotpoint}}
\put(568,211){\usebox{\plotpoint}}
\put(569,212){\usebox{\plotpoint}}
\put(570,213){\usebox{\plotpoint}}
\put(571,214){\usebox{\plotpoint}}
\put(572,215){\usebox{\plotpoint}}
\put(573,217){\usebox{\plotpoint}}
\put(574,218){\usebox{\plotpoint}}
\put(575,219){\usebox{\plotpoint}}
\put(576,220){\usebox{\plotpoint}}
\put(577,221){\usebox{\plotpoint}}
\put(578,223){\usebox{\plotpoint}}
\put(579,224){\usebox{\plotpoint}}
\put(580,225){\usebox{\plotpoint}}
\put(581,226){\usebox{\plotpoint}}
\put(582,227){\usebox{\plotpoint}}
\put(583,228){\usebox{\plotpoint}}
\put(584,230){\usebox{\plotpoint}}
\put(585,231){\usebox{\plotpoint}}
\put(586,232){\usebox{\plotpoint}}
\put(587,233){\usebox{\plotpoint}}
\put(588,235){\usebox{\plotpoint}}
\put(589,236){\usebox{\plotpoint}}
\put(590,237){\usebox{\plotpoint}}
\put(591,238){\usebox{\plotpoint}}
\put(592,240){\usebox{\plotpoint}}
\put(593,241){\usebox{\plotpoint}}
\put(594,242){\usebox{\plotpoint}}
\put(595,243){\usebox{\plotpoint}}
\put(596,245){\usebox{\plotpoint}}
\put(597,246){\usebox{\plotpoint}}
\put(598,247){\usebox{\plotpoint}}
\put(599,248){\usebox{\plotpoint}}
\put(600,249){\usebox{\plotpoint}}
\put(601,251){\usebox{\plotpoint}}
\put(602,252){\usebox{\plotpoint}}
\put(603,253){\usebox{\plotpoint}}
\put(604,254){\usebox{\plotpoint}}
\put(605,255){\usebox{\plotpoint}}
\put(606,256){\usebox{\plotpoint}}
\put(607,258){\usebox{\plotpoint}}
\put(608,259){\usebox{\plotpoint}}
\put(609,260){\usebox{\plotpoint}}
\put(610,261){\usebox{\plotpoint}}
\put(611,263){\usebox{\plotpoint}}
\put(612,264){\usebox{\plotpoint}}
\put(613,265){\usebox{\plotpoint}}
\put(614,266){\usebox{\plotpoint}}
\put(615,268){\usebox{\plotpoint}}
\put(616,269){\usebox{\plotpoint}}
\put(617,270){\usebox{\plotpoint}}
\put(618,271){\usebox{\plotpoint}}
\put(619,273){\usebox{\plotpoint}}
\put(620,274){\usebox{\plotpoint}}
\put(621,275){\usebox{\plotpoint}}
\put(622,276){\usebox{\plotpoint}}
\put(623,278){\usebox{\plotpoint}}
\put(624,279){\usebox{\plotpoint}}
\put(625,280){\usebox{\plotpoint}}
\put(626,281){\usebox{\plotpoint}}
\put(627,283){\usebox{\plotpoint}}
\put(628,284){\usebox{\plotpoint}}
\put(629,285){\usebox{\plotpoint}}
\put(630,286){\usebox{\plotpoint}}
\put(631,288){\usebox{\plotpoint}}
\put(632,289){\usebox{\plotpoint}}
\put(633,290){\usebox{\plotpoint}}
\put(634,291){\usebox{\plotpoint}}
\put(635,293){\usebox{\plotpoint}}
\put(636,294){\usebox{\plotpoint}}
\put(637,295){\usebox{\plotpoint}}
\put(638,296){\usebox{\plotpoint}}
\put(639,298){\usebox{\plotpoint}}
\put(640,299){\usebox{\plotpoint}}
\put(641,300){\usebox{\plotpoint}}
\put(642,301){\usebox{\plotpoint}}
\put(643,303){\usebox{\plotpoint}}
\put(644,304){\usebox{\plotpoint}}
\put(645,305){\usebox{\plotpoint}}
\put(646,306){\usebox{\plotpoint}}
\put(647,308){\usebox{\plotpoint}}
\put(648,309){\usebox{\plotpoint}}
\put(649,310){\usebox{\plotpoint}}
\put(650,311){\usebox{\plotpoint}}
\put(651,313){\usebox{\plotpoint}}
\put(652,314){\usebox{\plotpoint}}
\put(653,315){\usebox{\plotpoint}}
\put(654,316){\usebox{\plotpoint}}
\put(655,318){\usebox{\plotpoint}}
\put(656,319){\usebox{\plotpoint}}
\put(657,320){\usebox{\plotpoint}}
\put(658,322){\usebox{\plotpoint}}
\put(659,323){\usebox{\plotpoint}}
\put(660,324){\usebox{\plotpoint}}
\put(661,326){\usebox{\plotpoint}}
\put(662,327){\usebox{\plotpoint}}
\put(663,328){\usebox{\plotpoint}}
\put(664,330){\usebox{\plotpoint}}
\put(665,331){\usebox{\plotpoint}}
\put(666,332){\usebox{\plotpoint}}
\put(667,334){\rule[-0.175pt]{0.350pt}{0.350pt}}
\put(668,335){\rule[-0.175pt]{0.350pt}{0.350pt}}
\put(669,336){\rule[-0.175pt]{0.350pt}{0.350pt}}
\put(670,338){\rule[-0.175pt]{0.350pt}{0.350pt}}
\put(671,339){\rule[-0.175pt]{0.350pt}{0.350pt}}
\put(672,341){\rule[-0.175pt]{0.350pt}{0.350pt}}
\put(673,342){\rule[-0.175pt]{0.350pt}{0.350pt}}
\put(674,344){\rule[-0.175pt]{0.350pt}{0.350pt}}
\put(675,345){\rule[-0.175pt]{0.350pt}{0.350pt}}
\put(676,347){\rule[-0.175pt]{0.350pt}{0.350pt}}
\put(677,348){\rule[-0.175pt]{0.350pt}{0.350pt}}
\put(678,350){\usebox{\plotpoint}}
\put(679,351){\usebox{\plotpoint}}
\put(680,352){\usebox{\plotpoint}}
\put(681,354){\usebox{\plotpoint}}
\put(682,355){\usebox{\plotpoint}}
\put(683,356){\usebox{\plotpoint}}
\put(684,358){\usebox{\plotpoint}}
\put(685,359){\usebox{\plotpoint}}
\put(686,360){\usebox{\plotpoint}}
\put(687,362){\usebox{\plotpoint}}
\put(688,363){\usebox{\plotpoint}}
\put(689,364){\usebox{\plotpoint}}
\put(690,366){\usebox{\plotpoint}}
\put(691,367){\usebox{\plotpoint}}
\put(692,368){\usebox{\plotpoint}}
\put(693,369){\usebox{\plotpoint}}
\put(694,371){\usebox{\plotpoint}}
\put(695,372){\usebox{\plotpoint}}
\put(696,373){\usebox{\plotpoint}}
\put(697,374){\usebox{\plotpoint}}
\put(698,376){\usebox{\plotpoint}}
\put(699,377){\usebox{\plotpoint}}
\put(700,378){\usebox{\plotpoint}}
\put(701,379){\usebox{\plotpoint}}
\put(702,381){\usebox{\plotpoint}}
\put(703,382){\usebox{\plotpoint}}
\put(704,383){\usebox{\plotpoint}}
\put(705,385){\usebox{\plotpoint}}
\put(706,386){\usebox{\plotpoint}}
\put(707,387){\usebox{\plotpoint}}
\put(708,389){\usebox{\plotpoint}}
\put(709,390){\usebox{\plotpoint}}
\put(710,391){\usebox{\plotpoint}}
\put(711,393){\usebox{\plotpoint}}
\put(712,394){\usebox{\plotpoint}}
\put(713,395){\usebox{\plotpoint}}
\put(714,397){\usebox{\plotpoint}}
\put(715,398){\usebox{\plotpoint}}
\put(716,399){\usebox{\plotpoint}}
\put(717,401){\usebox{\plotpoint}}
\put(718,402){\usebox{\plotpoint}}
\put(719,403){\usebox{\plotpoint}}
\put(720,405){\usebox{\plotpoint}}
\put(721,406){\usebox{\plotpoint}}
\put(722,407){\usebox{\plotpoint}}
\put(723,409){\usebox{\plotpoint}}
\put(724,410){\usebox{\plotpoint}}
\put(725,411){\usebox{\plotpoint}}
\put(726,413){\usebox{\plotpoint}}
\put(727,414){\usebox{\plotpoint}}
\put(728,415){\usebox{\plotpoint}}
\put(729,417){\usebox{\plotpoint}}
\put(730,418){\usebox{\plotpoint}}
\put(731,419){\usebox{\plotpoint}}
\put(732,421){\usebox{\plotpoint}}
\put(733,422){\usebox{\plotpoint}}
\put(734,423){\usebox{\plotpoint}}
\put(735,425){\usebox{\plotpoint}}
\put(736,426){\usebox{\plotpoint}}
\put(737,427){\usebox{\plotpoint}}
\put(738,429){\usebox{\plotpoint}}
\put(739,430){\usebox{\plotpoint}}
\put(740,431){\usebox{\plotpoint}}
\put(741,433){\usebox{\plotpoint}}
\put(742,434){\usebox{\plotpoint}}
\put(743,435){\usebox{\plotpoint}}
\put(744,437){\usebox{\plotpoint}}
\put(745,438){\usebox{\plotpoint}}
\put(746,439){\usebox{\plotpoint}}
\put(747,441){\usebox{\plotpoint}}
\put(748,442){\usebox{\plotpoint}}
\put(749,444){\usebox{\plotpoint}}
\put(750,445){\usebox{\plotpoint}}
\put(751,446){\usebox{\plotpoint}}
\put(752,447){\usebox{\plotpoint}}
\put(753,449){\usebox{\plotpoint}}
\put(754,450){\usebox{\plotpoint}}
\put(755,451){\usebox{\plotpoint}}
\put(756,452){\usebox{\plotpoint}}
\put(757,454){\usebox{\plotpoint}}
\put(758,455){\usebox{\plotpoint}}
\put(759,456){\usebox{\plotpoint}}
\put(760,457){\usebox{\plotpoint}}
\put(761,459){\usebox{\plotpoint}}
\put(762,460){\usebox{\plotpoint}}
\put(763,461){\usebox{\plotpoint}}
\put(764,462){\usebox{\plotpoint}}
\put(765,464){\usebox{\plotpoint}}
\put(766,465){\usebox{\plotpoint}}
\put(767,466){\usebox{\plotpoint}}
\put(768,467){\usebox{\plotpoint}}
\put(769,469){\usebox{\plotpoint}}
\put(770,470){\usebox{\plotpoint}}
\put(771,471){\usebox{\plotpoint}}
\put(772,472){\usebox{\plotpoint}}
\put(773,474){\usebox{\plotpoint}}
\put(774,475){\usebox{\plotpoint}}
\put(775,476){\usebox{\plotpoint}}
\put(776,477){\usebox{\plotpoint}}
\put(777,478){\usebox{\plotpoint}}
\put(778,479){\usebox{\plotpoint}}
\put(779,480){\usebox{\plotpoint}}
\put(780,482){\usebox{\plotpoint}}
\put(781,483){\usebox{\plotpoint}}
\put(782,484){\usebox{\plotpoint}}
\put(783,485){\usebox{\plotpoint}}
\put(784,486){\usebox{\plotpoint}}
\put(785,487){\usebox{\plotpoint}}
\put(786,489){\usebox{\plotpoint}}
\put(787,490){\usebox{\plotpoint}}
\put(788,491){\usebox{\plotpoint}}
\put(789,492){\usebox{\plotpoint}}
\put(790,493){\usebox{\plotpoint}}
\put(791,494){\usebox{\plotpoint}}
\put(792,496){\usebox{\plotpoint}}
\put(793,497){\usebox{\plotpoint}}
\put(794,498){\usebox{\plotpoint}}
\put(795,499){\usebox{\plotpoint}}
\put(796,500){\usebox{\plotpoint}}
\put(797,501){\usebox{\plotpoint}}
\put(798,503){\usebox{\plotpoint}}
\put(799,504){\usebox{\plotpoint}}
\put(800,505){\usebox{\plotpoint}}
\put(801,506){\usebox{\plotpoint}}
\put(802,507){\usebox{\plotpoint}}
\put(803,508){\usebox{\plotpoint}}
\put(804,509){\usebox{\plotpoint}}
\put(805,510){\usebox{\plotpoint}}
\put(806,511){\usebox{\plotpoint}}
\put(807,512){\usebox{\plotpoint}}
\put(808,513){\usebox{\plotpoint}}
\put(809,515){\usebox{\plotpoint}}
\put(810,516){\usebox{\plotpoint}}
\put(811,517){\usebox{\plotpoint}}
\put(812,518){\usebox{\plotpoint}}
\put(813,519){\usebox{\plotpoint}}
\put(814,520){\usebox{\plotpoint}}
\put(815,522){\usebox{\plotpoint}}
\put(816,523){\usebox{\plotpoint}}
\put(817,524){\usebox{\plotpoint}}
\put(818,525){\usebox{\plotpoint}}
\put(819,526){\usebox{\plotpoint}}
\put(820,528){\usebox{\plotpoint}}
\put(821,529){\usebox{\plotpoint}}
\put(822,530){\usebox{\plotpoint}}
\put(823,531){\usebox{\plotpoint}}
\put(824,532){\usebox{\plotpoint}}
\put(825,533){\usebox{\plotpoint}}
\put(826,534){\usebox{\plotpoint}}
\put(827,535){\usebox{\plotpoint}}
\put(828,536){\usebox{\plotpoint}}
\put(829,537){\usebox{\plotpoint}}
\put(830,538){\usebox{\plotpoint}}
\put(831,539){\usebox{\plotpoint}}
\put(832,540){\usebox{\plotpoint}}
\put(833,541){\usebox{\plotpoint}}
\put(834,542){\usebox{\plotpoint}}
\put(835,543){\usebox{\plotpoint}}
\put(836,544){\usebox{\plotpoint}}
\put(837,545){\usebox{\plotpoint}}
\put(838,546){\usebox{\plotpoint}}
\put(839,547){\usebox{\plotpoint}}
\put(840,548){\usebox{\plotpoint}}
\put(841,549){\usebox{\plotpoint}}
\put(842,550){\usebox{\plotpoint}}
\put(843,551){\usebox{\plotpoint}}
\put(845,552){\usebox{\plotpoint}}
\put(854,561){\usebox{\plotpoint}}
\put(855,562){\usebox{\plotpoint}}
\put(857,563){\usebox{\plotpoint}}
\put(858,564){\usebox{\plotpoint}}
\put(859,565){\usebox{\plotpoint}}
\put(861,566){\usebox{\plotpoint}}
\put(862,567){\usebox{\plotpoint}}
\put(863,568){\usebox{\plotpoint}}
\put(865,569){\usebox{\plotpoint}}
\put(866,570){\usebox{\plotpoint}}
\put(867,571){\usebox{\plotpoint}}
\put(869,572){\usebox{\plotpoint}}
\put(870,573){\usebox{\plotpoint}}
\put(871,574){\usebox{\plotpoint}}
\put(872,575){\usebox{\plotpoint}}
\put(874,576){\usebox{\plotpoint}}
\put(875,577){\usebox{\plotpoint}}
\put(876,578){\usebox{\plotpoint}}
\put(877,579){\usebox{\plotpoint}}
\put(878,580){\usebox{\plotpoint}}
\put(880,581){\usebox{\plotpoint}}
\put(881,582){\usebox{\plotpoint}}
\put(882,583){\usebox{\plotpoint}}
\put(884,584){\usebox{\plotpoint}}
\put(885,585){\usebox{\plotpoint}}
\put(886,586){\usebox{\plotpoint}}
\put(888,587){\usebox{\plotpoint}}
\put(889,588){\usebox{\plotpoint}}
\put(891,589){\rule[-0.175pt]{0.361pt}{0.350pt}}
\put(892,590){\rule[-0.175pt]{0.361pt}{0.350pt}}
\put(894,591){\rule[-0.175pt]{0.361pt}{0.350pt}}
\put(895,592){\rule[-0.175pt]{0.361pt}{0.350pt}}
\put(897,593){\rule[-0.175pt]{0.361pt}{0.350pt}}
\put(898,594){\rule[-0.175pt]{0.361pt}{0.350pt}}
\put(900,595){\rule[-0.175pt]{0.361pt}{0.350pt}}
\put(901,596){\rule[-0.175pt]{0.361pt}{0.350pt}}
\put(903,597){\rule[-0.175pt]{0.361pt}{0.350pt}}
\put(904,598){\rule[-0.175pt]{0.361pt}{0.350pt}}
\put(906,599){\rule[-0.175pt]{0.361pt}{0.350pt}}
\put(907,600){\rule[-0.175pt]{0.361pt}{0.350pt}}
\put(909,601){\rule[-0.175pt]{0.361pt}{0.350pt}}
\put(910,602){\rule[-0.175pt]{0.361pt}{0.350pt}}
\put(912,603){\rule[-0.175pt]{0.361pt}{0.350pt}}
\put(913,604){\rule[-0.175pt]{0.361pt}{0.350pt}}
\put(915,605){\rule[-0.175pt]{0.413pt}{0.350pt}}
\put(916,606){\rule[-0.175pt]{0.413pt}{0.350pt}}
\put(918,607){\rule[-0.175pt]{0.413pt}{0.350pt}}
\put(920,608){\rule[-0.175pt]{0.413pt}{0.350pt}}
\put(921,609){\rule[-0.175pt]{0.413pt}{0.350pt}}
\put(923,610){\rule[-0.175pt]{0.413pt}{0.350pt}}
\put(925,611){\rule[-0.175pt]{0.413pt}{0.350pt}}
\put(927,612){\rule[-0.175pt]{0.482pt}{0.350pt}}
\put(929,613){\rule[-0.175pt]{0.482pt}{0.350pt}}
\put(931,614){\rule[-0.175pt]{0.482pt}{0.350pt}}
\put(933,615){\rule[-0.175pt]{0.482pt}{0.350pt}}
\put(935,616){\rule[-0.175pt]{0.482pt}{0.350pt}}
\put(937,617){\rule[-0.175pt]{0.482pt}{0.350pt}}
\put(939,618){\rule[-0.175pt]{0.482pt}{0.350pt}}
\put(941,619){\rule[-0.175pt]{0.482pt}{0.350pt}}
\put(943,620){\rule[-0.175pt]{0.482pt}{0.350pt}}
\put(945,621){\rule[-0.175pt]{0.482pt}{0.350pt}}
\put(947,622){\rule[-0.175pt]{0.482pt}{0.350pt}}
\put(949,623){\rule[-0.175pt]{0.482pt}{0.350pt}}
\put(951,624){\rule[-0.175pt]{0.442pt}{0.350pt}}
\put(952,625){\rule[-0.175pt]{0.442pt}{0.350pt}}
\put(954,626){\rule[-0.175pt]{0.442pt}{0.350pt}}
\put(956,627){\rule[-0.175pt]{0.442pt}{0.350pt}}
\put(958,628){\rule[-0.175pt]{0.442pt}{0.350pt}}
\put(960,629){\rule[-0.175pt]{0.442pt}{0.350pt}}
\put(961,630){\rule[-0.175pt]{0.578pt}{0.350pt}}
\put(964,631){\rule[-0.175pt]{0.578pt}{0.350pt}}
\put(966,632){\rule[-0.175pt]{0.578pt}{0.350pt}}
\put(969,633){\rule[-0.175pt]{0.578pt}{0.350pt}}
\put(971,634){\rule[-0.175pt]{0.578pt}{0.350pt}}
\put(974,635){\rule[-0.175pt]{0.482pt}{0.350pt}}
\put(976,636){\rule[-0.175pt]{0.482pt}{0.350pt}}
\put(978,637){\rule[-0.175pt]{0.482pt}{0.350pt}}
\put(980,638){\rule[-0.175pt]{0.482pt}{0.350pt}}
\put(982,639){\rule[-0.175pt]{0.482pt}{0.350pt}}
\put(984,640){\rule[-0.175pt]{0.482pt}{0.350pt}}
\put(986,641){\rule[-0.175pt]{0.578pt}{0.350pt}}
\put(988,642){\rule[-0.175pt]{0.578pt}{0.350pt}}
\put(990,643){\rule[-0.175pt]{0.578pt}{0.350pt}}
\put(993,644){\rule[-0.175pt]{0.578pt}{0.350pt}}
\put(995,645){\rule[-0.175pt]{0.578pt}{0.350pt}}
\put(998,646){\rule[-0.175pt]{0.578pt}{0.350pt}}
\put(1000,647){\rule[-0.175pt]{0.578pt}{0.350pt}}
\put(1002,648){\rule[-0.175pt]{0.578pt}{0.350pt}}
\put(1005,649){\rule[-0.175pt]{0.578pt}{0.350pt}}
\put(1007,650){\rule[-0.175pt]{0.578pt}{0.350pt}}
\put(1010,651){\rule[-0.175pt]{0.482pt}{0.350pt}}
\put(1012,652){\rule[-0.175pt]{0.482pt}{0.350pt}}
\put(1014,653){\rule[-0.175pt]{0.482pt}{0.350pt}}
\put(1016,654){\rule[-0.175pt]{0.482pt}{0.350pt}}
\put(1018,655){\rule[-0.175pt]{0.482pt}{0.350pt}}
\put(1020,656){\rule[-0.175pt]{0.482pt}{0.350pt}}
\put(1022,657){\rule[-0.175pt]{0.530pt}{0.350pt}}
\put(1024,658){\rule[-0.175pt]{0.530pt}{0.350pt}}
\put(1026,659){\rule[-0.175pt]{0.530pt}{0.350pt}}
\put(1028,660){\rule[-0.175pt]{0.530pt}{0.350pt}}
\put(1030,661){\rule[-0.175pt]{0.530pt}{0.350pt}}
\put(1032,662){\rule[-0.175pt]{0.482pt}{0.350pt}}
\put(1035,663){\rule[-0.175pt]{0.482pt}{0.350pt}}
\put(1037,664){\rule[-0.175pt]{0.482pt}{0.350pt}}
\put(1039,665){\rule[-0.175pt]{0.482pt}{0.350pt}}
\put(1041,666){\rule[-0.175pt]{0.482pt}{0.350pt}}
\put(1043,667){\rule[-0.175pt]{0.482pt}{0.350pt}}
\put(1045,668){\rule[-0.175pt]{0.482pt}{0.350pt}}
\put(1047,669){\rule[-0.175pt]{0.482pt}{0.350pt}}
\put(1049,670){\rule[-0.175pt]{0.482pt}{0.350pt}}
\put(1051,671){\rule[-0.175pt]{0.482pt}{0.350pt}}
\put(1053,672){\rule[-0.175pt]{0.482pt}{0.350pt}}
\put(1055,673){\rule[-0.175pt]{0.482pt}{0.350pt}}
\put(1057,674){\rule[-0.175pt]{0.482pt}{0.350pt}}
\put(1059,675){\rule[-0.175pt]{0.482pt}{0.350pt}}
\put(1061,676){\rule[-0.175pt]{0.482pt}{0.350pt}}
\put(1063,677){\rule[-0.175pt]{0.482pt}{0.350pt}}
\put(1065,678){\rule[-0.175pt]{0.482pt}{0.350pt}}
\put(1067,679){\rule[-0.175pt]{0.482pt}{0.350pt}}
\put(1069,680){\rule[-0.175pt]{0.482pt}{0.350pt}}
\put(1071,681){\rule[-0.175pt]{0.482pt}{0.350pt}}
\put(1073,682){\rule[-0.175pt]{0.482pt}{0.350pt}}
\put(1075,683){\rule[-0.175pt]{0.482pt}{0.350pt}}
\put(1077,684){\rule[-0.175pt]{0.482pt}{0.350pt}}
\put(1079,685){\rule[-0.175pt]{0.482pt}{0.350pt}}
\put(1081,686){\rule[-0.175pt]{0.482pt}{0.350pt}}
\put(1083,687){\rule[-0.175pt]{0.482pt}{0.350pt}}
\put(1085,688){\rule[-0.175pt]{0.482pt}{0.350pt}}
\put(1087,689){\rule[-0.175pt]{0.482pt}{0.350pt}}
\put(1089,690){\rule[-0.175pt]{0.482pt}{0.350pt}}
\put(1091,691){\rule[-0.175pt]{0.482pt}{0.350pt}}
\put(1093,692){\rule[-0.175pt]{0.413pt}{0.350pt}}
\put(1094,693){\rule[-0.175pt]{0.413pt}{0.350pt}}
\put(1096,694){\rule[-0.175pt]{0.413pt}{0.350pt}}
\put(1098,695){\rule[-0.175pt]{0.413pt}{0.350pt}}
\put(1099,696){\rule[-0.175pt]{0.413pt}{0.350pt}}
\put(1101,697){\rule[-0.175pt]{0.413pt}{0.350pt}}
\put(1103,698){\rule[-0.175pt]{0.413pt}{0.350pt}}
\put(1104,699){\rule[-0.175pt]{0.379pt}{0.350pt}}
\put(1106,700){\rule[-0.175pt]{0.379pt}{0.350pt}}
\put(1108,701){\rule[-0.175pt]{0.379pt}{0.350pt}}
\put(1109,702){\rule[-0.175pt]{0.379pt}{0.350pt}}
\put(1111,703){\rule[-0.175pt]{0.379pt}{0.350pt}}
\put(1112,704){\rule[-0.175pt]{0.379pt}{0.350pt}}
\put(1114,705){\rule[-0.175pt]{0.379pt}{0.350pt}}
\put(1115,706){\rule[-0.175pt]{0.361pt}{0.350pt}}
\put(1117,707){\rule[-0.175pt]{0.361pt}{0.350pt}}
\put(1119,708){\rule[-0.175pt]{0.361pt}{0.350pt}}
\put(1120,709){\rule[-0.175pt]{0.361pt}{0.350pt}}
\put(1122,710){\rule[-0.175pt]{0.361pt}{0.350pt}}
\put(1123,711){\rule[-0.175pt]{0.361pt}{0.350pt}}
\put(1125,712){\rule[-0.175pt]{0.361pt}{0.350pt}}
\put(1126,713){\rule[-0.175pt]{0.361pt}{0.350pt}}
\put(1128,714){\rule[-0.175pt]{0.361pt}{0.350pt}}
\put(1129,715){\rule[-0.175pt]{0.361pt}{0.350pt}}
\put(1131,716){\rule[-0.175pt]{0.361pt}{0.350pt}}
\put(1132,717){\rule[-0.175pt]{0.361pt}{0.350pt}}
\put(1134,718){\rule[-0.175pt]{0.361pt}{0.350pt}}
\put(1135,719){\rule[-0.175pt]{0.361pt}{0.350pt}}
\put(1137,720){\rule[-0.175pt]{0.361pt}{0.350pt}}
\put(1138,721){\rule[-0.175pt]{0.361pt}{0.350pt}}
\put(1140,722){\rule[-0.175pt]{0.361pt}{0.350pt}}
\put(1141,723){\rule[-0.175pt]{0.361pt}{0.350pt}}
\put(1143,724){\rule[-0.175pt]{0.361pt}{0.350pt}}
\put(1144,725){\rule[-0.175pt]{0.361pt}{0.350pt}}
\put(1146,726){\rule[-0.175pt]{0.361pt}{0.350pt}}
\put(1147,727){\rule[-0.175pt]{0.361pt}{0.350pt}}
\put(1149,728){\rule[-0.175pt]{0.361pt}{0.350pt}}
\put(1150,729){\rule[-0.175pt]{0.361pt}{0.350pt}}
\put(1152,730){\rule[-0.175pt]{0.361pt}{0.350pt}}
\put(1153,731){\rule[-0.175pt]{0.361pt}{0.350pt}}
\put(1155,732){\rule[-0.175pt]{0.361pt}{0.350pt}}
\put(1156,733){\rule[-0.175pt]{0.361pt}{0.350pt}}
\put(1158,734){\rule[-0.175pt]{0.361pt}{0.350pt}}
\put(1159,735){\rule[-0.175pt]{0.361pt}{0.350pt}}
\put(1161,736){\rule[-0.175pt]{0.361pt}{0.350pt}}
\put(1162,737){\rule[-0.175pt]{0.361pt}{0.350pt}}
\put(1164,738){\usebox{\plotpoint}}
\put(1165,739){\usebox{\plotpoint}}
\put(1166,740){\usebox{\plotpoint}}
\put(1168,741){\usebox{\plotpoint}}
\put(1169,742){\usebox{\plotpoint}}
\put(1170,743){\usebox{\plotpoint}}
\put(1172,744){\usebox{\plotpoint}}
\put(1173,745){\usebox{\plotpoint}}
\put(1174,746){\usebox{\plotpoint}}
\put(1176,747){\usebox{\plotpoint}}
\put(1177,748){\usebox{\plotpoint}}
\put(1178,749){\usebox{\plotpoint}}
\put(1179,750){\usebox{\plotpoint}}
\put(1180,751){\usebox{\plotpoint}}
\put(1182,752){\usebox{\plotpoint}}
\put(1183,753){\usebox{\plotpoint}}
\put(1184,754){\usebox{\plotpoint}}
\put(1185,755){\usebox{\plotpoint}}
\put(1186,756){\usebox{\plotpoint}}
\put(1188,757){\usebox{\plotpoint}}
\put(1189,758){\usebox{\plotpoint}}
\put(1191,759){\usebox{\plotpoint}}
\put(1192,760){\usebox{\plotpoint}}
\put(1193,761){\usebox{\plotpoint}}
\put(1195,762){\usebox{\plotpoint}}
\put(1196,763){\usebox{\plotpoint}}
\put(1197,764){\usebox{\plotpoint}}
\put(1199,765){\usebox{\plotpoint}}
\put(1200,766){\usebox{\plotpoint}}
\put(1201,767){\usebox{\plotpoint}}
\put(1203,768){\usebox{\plotpoint}}
\put(1204,769){\usebox{\plotpoint}}
\put(1205,770){\usebox{\plotpoint}}
\put(1207,771){\usebox{\plotpoint}}
\put(1208,772){\usebox{\plotpoint}}
\put(1209,773){\usebox{\plotpoint}}
\put(1211,774){\usebox{\plotpoint}}
\put(1212,775){\usebox{\plotpoint}}
\put(1213,776){\usebox{\plotpoint}}
\put(1215,777){\usebox{\plotpoint}}
\put(1216,778){\usebox{\plotpoint}}
\put(1217,779){\usebox{\plotpoint}}
\put(1219,780){\usebox{\plotpoint}}
\put(1220,781){\usebox{\plotpoint}}
\put(1221,782){\usebox{\plotpoint}}
\put(1223,783){\usebox{\plotpoint}}
\put(1224,784){\usebox{\plotpoint}}
\put(1225,785){\usebox{\plotpoint}}
\put(1226,786){\usebox{\plotpoint}}
\end{picture}

\end{center}
\caption{}\label{psi-pic}
\end{figure}

\medskip
\noindent
Following Rosenlicht~\cite{Rosenlicht83}, we put 
$$   \Psi\ =\ \Psi_K\ :=\ \{\gamma^\dagger:\  \gamma\in \Gamma^{\ne}\}.$$ 
Then
$\Psi < (\Gamma^{>})'$. In the rest of this subsection we assume that $K$ is an $H$-field. Then exactly one of the following holds:
\begin{description}
\item[Case 1]  $\Psi < \gamma < (\Gamma^{>})'$ for some (necessarily unique) $\gamma$;
\item[Case 2]  $\Psi$ has a largest element;
\item[Case 3]  $\sup \Psi$ does not exist; equivalently, $\Gamma=(\Gamma^{\ne})'$.
\end{description} 
If $K=C$ we are in Case~1, with $\gamma=0$; the Laurent series field 
$\R \((  x^{-1}  \)) $ falls under Case~2, and Liouville closed $H$-fields 
under Case~3. 
In Case~1 there are two Liouville closures of $K$; in Case~2 
there is only one, but Case~3 requires finer distinctions for a 
definite answer. 
We now explain this in more detail.

Suppose $K$ falls under Case~1.
Then the element $\gamma$ is called a {\bf gap}, and there are two  
ways to remove the gap: with $v(a)=\gamma$, we have an $H$-field extension
$K(y_1)$ with $y_1\prec 1$ and $y_1'=a$, and we also have an 
$H$-field extension $K(y_2)$ with $0\ne y_2\prec 1$ and $y_2^{\dagger}=a$. Both
of these extensions fall under Case~2, and they are incompatible in the sense 
that they cannot be embedded over $K$ into a common $H$-field extension
of $K$. (Any Liouville closed extension of $K$, however, contains either a
copy of $K(y_1)$ or a copy of $K(y_2)$.) Instead of ``$K$ falls under Case~1''
we say ``$K$ has a gap.'' 

Suppose $K$ falls under Case~2. Then so does every 
differential-algebraic $H$-field extension of $K$ that is finitely generated
over $K$ as a differential field. Thus it takes an infinite generation process
to construct the (unique) Liouville closure of $K$. 
An $H$-field falling under Case~3 is said to admit 
{\bf asymptotic integration}. This is because Case~3 is equivalent to 
having for each non-zero $a$ in the field an element $y$ in the field such
that $y'\sim a$.

\subsection{Asymptotic couples.}\label{sec:asymptotic couples} Let $K$ be an asymptotic field. 
The ordered group $\Gamma=v(K^\times)$ equipped with the function
$$\gamma \mapsto \gamma^\dagger\colon \Gamma^{\ne} \to \Gamma$$ is an
{\bf asymptotic couple} in the terminology of Rosenlicht \cite{Rosenlicht79, Rosenlicht80, Rosenlicht81}, who
proved the first non-trivial facts about them as structures in their own
right, independent of their origin in Hardy fields. Indeed, this function $\gamma \mapsto \gamma^\dagger$ 
has rather nice properties: it is a valuation on the 
abelian group $\Gamma$: for all $\alpha,\beta\in\Gamma$ (and $0^{\dagger}:= \infty > \Gamma$),
\begin{enumerate}
\item[(i)] $(\alpha+\beta)^\dagger \geq\min(\alpha^\dagger,\beta^\dagger)$, 
\item[(ii)] $(-\alpha)^\dagger=\alpha^\dagger$. 
\end{enumerate}
If $K$ is moreover an $H$-field, then this valuation is compatible with the 
group ordering in the sense that for all $\alpha,\beta\in \Gamma$, 
\begin{equation}\label{eq:H}\tag{H}
0 < \alpha \le \beta\ \Longrightarrow\ \alpha^\dagger \ge \beta^\dagger.
\end{equation} 
The trichotomy from the previous section
holds for all asymptotic couples satisfying \eqref{eq:H}; see \cite{AvdD1}.  
The asymptotic couple of $\T$ has a good model theory:
It allows elimination of quantifiers in its natural 
language augmented by a predicate for the subset $\Psi$ of $\Gamma$, and its
theory is axiomatized by adding to Rosenlicht's axioms\footnote{Formally, an asymptotic couple is an 
ordered abelian group $\Gamma$ equipped with a valuation 
$\psi\colon \Gamma^{\ne} \to \Gamma$ such that $\psi(\alpha) < \beta+\psi(\beta)$
for all $\alpha,\beta\in \Gamma^{>}$.} for asymptotic couples
the following requirements: 
\begin{enumerate}
\item[(i)] divisibility of the underlying abelian group;
\item[(ii)] compatibility with the ordering as in \eqref{eq:H} above;
\item[(iii)] $\Psi$ is downward closed, $0\in \Psi$,  and $\Psi$ has no maximum;
\item[(iv)] there is no gap.
\end{enumerate}
This result is in \cite{AvdD}, which proves also the weak o-minimality of 
the asymptotic couple of $\T$. We do not want to create the impression that the structure induced by $\T$ on its value group $\Gamma$ is just that of an asymptotic couple: $\Gamma$ is also a vector space over the constant field 
$\R$: $r\alpha=\beta$ for $r\in \R$ and $\alpha,\beta\in \Gamma^{\ne}$ whenever $ra^\dagger=b^\dagger$ and $a,b\in \T$, $va=\alpha, vb=\beta$. This vector space structure is also accounted for in \cite{AvdD}. Moreover, these facts about
$\T$ hold
for any Liouville closed $H$-field $K$ with small derivation (with $\R$ 
replaced by $C$).

\section{New  Results}\label{sec:new results}

\noindent
The above material raises some issues which turn out to be related.
First, no $H$-subfield $K$ of $\T$ with $\Gamma\ne \{0\}$ has a gap. Even
to construct a {\em Hardy field\/} with a gap and $\Gamma\ne \{0\}$ 
takes effort.
Nevertheless, the model theory of asymptotic couples strongly suggests 
that $H$-fields with a gap should play a key role, and so the question arises
how a given $H$-field can be extended to one with a gap.
The analogous issue for asymptotic couples is easy, but
we only managed to show rather recently that every $H$-field can be 
extended to one with a gap. This is discussed in more detail in Section~\ref{sec:important pc-sequence}.

Recall: a valued field is {\em maximal\/} if 
it has no proper immediate extension; this is equivalent to the more 
geometric notion of {\em spherically complete}. For example, Hahn fields 
are maximal. Decisive results in the model theory 
of maximal valued fields are due to
Ax~\&~Kochen~\cite{AK65} and Er{\v{s}}ov~\cite{Ersov}. Among other 
things they showed that {\em henselian\/} 
is the exact first-order counterpart of {\em maximal\/}, at least 
in equicharacteristic zero. 
In later extensions
(Scanlon's valued differential fields in~\cite{Scanlon} and the
valued difference fields from~\cite{Azgin-vdDries,BMS}), the natural models are 
still maximal. Here and below, ``maximal'' means ``maximal as a valued field''
even if the valued field in question has further structure like a derivation.

However, in our situation the expected natural models cannot be maximal:  
no maximal $H$-field can be Liouville closed, let
alone existentially closed. Maximal $H$-fields do nevertheless exist in 
abundance, and turn out to be a natural source for creating $H$-fields 
with a gap. It also remains true that immediate
extensions require close attention: $\T$ has
proper immediate $H$-field extensions that embed over $\T$ into an 
elementary extension of $\T$; see the proof of Proposition~\ref{nqel}. 
Thus we cannot bypass the immediate extensions of $\T$ 
in any model-theoretic account of $\T$ as we are aiming for.

\subsection{Immediate extensions of $H$-fields.}\label{immeh} 

\begin{theorem}\label{thimmax}
Every real closed $H$-field has an immediate maximal $H$-field extension.
\end{theorem}

This was not even known when the value group is $\Q$. A difference with the 
situation for 
valued fields of equicharacteristic $0$ (without derivation) 
is the lack of uniqueness of the maximal immediate extension. 
(The proof of Proposition~\ref{nqel} shows such non-uniqueness 
in the case of $\T$.) 

\medskip\noindent
Here are some comments on our proof of 
Theorem~\ref{thimmax}. First, this involves a change of derivation as follows. 
Let $K$ be a differential field with derivation~$\der$, and 
let $\phi\in K^\times$. Then we define $K^\phi$ to be 
the differential field obtained from~$K$ by taking
$\phi^{-1}\der$ as its derivation instead of $\der$. Then
the constant field~$C$ of $K$ is also the constant field of $K^{\phi}$, and so
$C\{Y\}$ is a common differential subring of $K\{Y\}$ and $K^{\phi}\{Y\}$.
Given a differential polynomial $P\in K\{Y\}$ we let 
$P^{\phi}\in K^{\phi}\{Y\}$ be the
result of rewriting $P$ in terms of the derivation $\phi^{-1}\der$, so
$P^{\phi}(y)=P(y)$ for all $y\in K$. (For example, $Y'^{\phi}=\phi Y'$ in
$K^{\phi}\{Y\}$.) This change of derivation is called 
{\bf compositional conjugation}. A suitable choice of $\phi$ can often
drastically simplify things. Also, if $K$ is an $H$-field and $\phi>0$, then
$K^{\phi}$ is still an $H$-field, with $\Psi_{K^{\phi}}=\Psi_K-v\phi$.

Next, given any valued differential field $K$, we extend its
valuation $v$ to a valuation on the domain $K\{Y\}$ of 
differential polynomials by 
$$vP:= \min\{va:\text{$a\in K$ is a coefficient of $P$}\}.$$ 
\noindent
Let now $K$ be a real closed $H$-field with 
value group $\Gamma\ne \{0\}$, and suppose first that
$K$ does not admit asymptotic integration.
Then $\sup \Psi$ exists in $\Gamma$, and  
by compositional conjugation we can arrange that
$\sup \Psi = 0$. One can show that then $K$ is {\bf flexible}, by
which we mean that it has the following 
property: 
for any $P\in K\{Y\}$ with $vP(0) > vP$ and any
$\gamma\in \Gamma^{>}$, the set $\{vP(y):\ y\in K,\ |vy|<\gamma\}$ is infinite.
This property then plays a key role in constructing 
an immediate maximal $H$-field extension of $K$. 
(It is worth mentioning that the notion of flexibility makes 
sense for any valued differential field. There are indeed 
other kinds of flexible valued differential fields such as those
considered in \cite{Scanlon} where this property can 
be used for similar ends.)

The case that the real closed $H$-field $K$ does admit asymptotic integration
is harder and uses compositional conjugation in a
more delicate way. We say more on this in the next subsection.

\subsection{The Newton polynomial.}\label{newtpo}

In this subsection $K$ is a real closed $H$-field with asymptotic integration.
To simulate the favorable case $\sup \Psi = 0$
from the previous subsection, we use compositional conjugation by
$\phi$ with $v\phi < (\Gamma^{>})'$ as large as possible.
Call $\phi\in K$ {\bf admissible} if $v\phi < (\Gamma^{>})'$.

\begin{theorem} Let $P\in K\{Y\}$, $P\ne 0$.  Then there is a differential
polynomial $N_P\in C\{Y\}$, $N_P\ne 0$, such that for all admissible 
$\phi\in K$ with sufficiently large $v\phi$ we have $a\in K^\times$ and $
R\in K^{\phi}\{Y\}$ with
$$ P^\phi = a N_P + R\ \text{ in } K^{\phi}\{Y\}, \qquad vR > va.$$
\end{theorem}

We call $N_P$ the {\bf Newton polynomial} of $P$. As described here, $N_P$
is only determined up to multiplication by an element of $C^\times$, but the key
fact is that $N_P$ is independent of the admissible $\phi$ for high enough 
$v\phi$.  We now have 
a modified version of the flexibility property of the previous subsection:
given any non-zero $P\in K\{Y\}$ with $N_P(0)=0$ and any
$\gamma\in \Gamma^{>}$, the set $\{vP(y):\ y\in K,\ |vy|<\gamma\}$ is infinite.
This can then be used to prove Theorem~\ref{thimmax} for 
real closed $H$-fields with asymptotic integration. 

\subsection{Newtonian $H$-fields and differential-henselian asymptotic fields.}\label{difhen}
An important fact about $\T$ from \cite{vdH} is that
if the Newton polynomial of $P\in \T\{Y\}$ has degree $1$, 
then $P$ has a zero in the valuation ring. Let us
define an $H$-field $K$ to be {\bf newtonian} if it 
is real closed, admits 
asymptotic integration, and every non-zero $P\in K\{Y\}$ whose Newton 
polynomial has degree $1$ has a zero in the valuation ring. 
Thus $\T$ is newtonian. A more basic example of a 
newtonian $H$-field is $\T_{\text{log}}$. It is easy to see that if $K$ 
is newtonian, then every linear differential equation 
$a_0y + a_1y' + \cdots + a_ny^{(n)}=b$ with   
$a_0,\dots, a_n, b\in K$, $a_n\ne 0$,
has a solution in $K$. 

If $K$ is a 
newtonian $H$-field, then so is each compositional conjugate
$K^{\phi}$ with $\phi>0$, and certain 
coarsenings of such compositional conjugates of $K$ are 
{\em differential-henselian} in the following sense. 
Let $K$ be any valued differential field with small derivation, that is,
$\der\smallo \subseteq \smallo$. It is not hard to see that then 
$\der\mathcal O \subseteq \mathcal O$, and so the residue field 
$\k=\mathcal O/\smallo$ is a differential field.
In the spirit of \cite{Scanlon} we define
$K$ to be {\bf differential-henselian} if 
\begin{list}{*}{\setlength\leftmargin{3em}}
\item[(DH1)] every linear differential equation 
$a_0y + a_1y' + \cdots + a_ny^{(n)}=b$ with   
$a_0,\dots, a_n, b\in \k$, $a_n\ne 0$,
has a solution in $\k$;
\item[(DH2)] for every $P\in \mathcal O\{Y\}$ with $vP_0>0 $ and
$vP_1=0$, there is $y\in \smallo$ such that $P(y)=0$.
\end{list}
Here $P_d$ is the homogeneous part of degree $d$ of $P$, so 
$$P_0=P(0), \qquad P_1 = \sum_i \frac{\partial P}{\partial Y^{(i)}}(0)Y^{(i)}.$$
\noindent
We now have an analogue of the familiar lifting of residue fields
in henselian valued fields of equicharacteristic zero: 
if $K$ is differential-henselian, 
then every maximal differential subfield of $\mathcal{O}$ maps isomorphically
(as a differential field) onto $\k$ under the residue map. 

\medskip\noindent
If $K$ is an 
$H$-field with $\der\smallo \subseteq \smallo$, then the derivation on its 
residue field $\k$ is trivial, so
(DH1) fails. To make the notion of 
differential-henselian relevant for $H$-fields we need to consider coarsenings:
Suppose $K$ is a newtonian $H$-field and 
$\der\smallo \subseteq \smallo$. Then the value group $\Gamma=v(K^\times)$ has 
a distinguished non-trivial convex subgroup
$$\Delta\ :=\ \{\gamma\in \Gamma:\ \gamma^{\dagger}>0\},$$
and $K$ with the coarsened valuation 
$v_{\Delta}\colon K^\times \to \Gamma/\Delta$ is differential-henselian. 
Moreover, by passing to suitable compositional conjugates
of $K$, we can make this distinguished non-trivial convex subgroup $\Delta$ 
as small as we like, and in this way we can make the coarsened valuation 
approximate the original valuation as close as needed. We call
$K$ with $v_{\Delta}$ the {\bf flattening of $K$}.

\medskip\noindent
These coarsenings are asymptotic differential fields, as defined in Section~\ref{sec:H-fields}.
Let us consider more generally any asymptotic differential field  
$K$ with $\der\smallo \subseteq \smallo$. Then ``differential-henselian''
does have some further general consequences:

\begin{lemma}\label{lemlindh} If $K$ is differential-henselian and 
$a_0,\dots, a_n, b\in K$, $a_n\ne 0$, then
$a_0y + a_1y' + \cdots + a_ny^{(n)}=b$ for some $y\in K$.   
\end{lemma} 

This has a useful sharper version where we assume that 
$a_0,\dots, a_n, b\in \mathcal{O}$ and $a_i\notin \smallo$ for some $i$, 
with the solution $y$ also required to be in $\mathcal{O}$.  

\begin{prop}\label{maxdifhen}
If $K$ is maximal as a valued field, and its differential residue field
$\k$ satisfies \textup{(DH1)}, then $K$ is differential-henselian. 
\end{prop}

While the AKE paradigm\footnote{``AKE'' stands for ``Ax-Kochen-Er\v{s}ov''.} 
does not apply directly to $H$-fields, it may well be relevant indirectly
by passing to coarsenings of compositional conjugates of $H$-fields.
Here we have of course in mind that  ``differential-henselian''
should take over the role of ``henselian'' in the AKE-theory.

It is worth mentioning that in dealing with a pc-sequence $(a_{\lambda})$ in
an $H$-field with asymptotic integration we can reduce to two very different
types of behavior: one kind of behavior is when $(a_{\lambda})$ 
is {\bf fluent}, that is, it remains a pc-sequence upon coarsening
the valuation by some non-trivial convex subgroup of the value group $\Gamma$, 
and the other type of behavior is when $(a_{\lambda})$ is {\bf jammed\/},
that is, for every  $\delta\in \Gamma^{>}$ there is
an index $\lambda_0$ such that
$|\gamma_{\mu}-\gamma_{\lambda}|<\delta$ for all $\mu > \lambda > \lambda_0$, where $\gamma_{\lambda}:= v(a_{\lambda+1}-a_{\lambda})$ for all $\lambda$.
In differential-henselian matters it is enough to deal with fluent pc-sequences.
Jammed pc-sequences are considered in 
Section~\ref{sec:important pc-sequence}.

\subsection{Consequences for existentially closed $H$-fields.}

Using the results above, some important facts about $\T$ can be shown to 
go through for existentially closed $H$-fields. Thus existentially 
closed $H$-fields are not only Liouville closed, but also 
newtonian. As to linear differential 
equations, let us go into a little more detail. 

\medskip\noindent
Let $K$ be a differential field, and consider a 
{\it linear differential operator\/}  
$$A=a_0 + a_1\der + \dots + a_n\der^n \qquad(a_0,\dots, a_n\in K)$$
over $K$; here $\der$ stands for the derivation operator on $K$. Then
$A$ defines a $C$-linear map $K \to K$. With composition as product operation,
these operators form a ring extension $K[\der]$ of $K$, with 
$\der  a=a\der+a'$ for $a\in K$.

\begin{theorem} If $K$ is an existentially closed $H$-field, $n\ge 1$, 
$a_n\ne 0$, then $A\colon K \to K$ is surjective, and $A$ is a product of
operators $a+b\der$ of order $1$ in $K[\imag][\der]$  \textup{(}and thus a product of
order $1$ and order $2$ operators in $K[\der]$\textup{)}. 
\end{theorem}

\subsection{The Equalizer Theorem.} This is an important technical tool, 
needed, for example, in proving Proposition~\ref{maxdifhen}.

\medskip\noindent
Let $K$ be a valued differential field with small derivation and value group
$\Gamma$. Let $P= P(Y)\in K\{Y\},\ P\ne 0$. Then we have for $g\in K^\times$ 
the non-zero differential polynomial $P(gY)\in K\{Y\}$, and it turns out that
its valuation $vP(gY)$ depends only on $vg$ (not on $g$). 
Thus $P$ induces a function
$$v_P\,\colon \ \Gamma\to \Gamma, \quad v_P(\gamma)\ :=\ vP(gY) \text{ for $g\in K^\times$ with $vg=\gamma$.}$$
Moreover, if $P(0)=0$, this function is strictly increasing. The function $v_P$ constrains the behavior of $vP(y)$ as $y$ ranges over $K^\times$:

\begin{lemma} If the derivation on the residue field $\k$ is non-trivial, then $v_P(\gamma)=\min\{vP(y):\ y\in K^\times,\ vy=\gamma\}$
for each $\gamma\in \Gamma$.
\end{lemma}

The following ``equalizer'' theorem lies much deeper:

\begin{theorem} Let $K$ be an asymptotic differential field with 
small derivation and divisible value group $\Gamma$.
Let $P\in K\{Y\}$, $P\ne 0$, be homogeneous of degree $d>0$. Then
$v_P\colon \Gamma \to \Gamma$ is a bijection. If also
$Q\in K\{Y\}$, $Q\ne 0$, is homogeneous of degree $e\ne d$, then
there is a unique $\gamma\in \Gamma$ with $v_P(\gamma)=v_Q(\gamma)$.
\end{theorem}

In combination with compositional conjugation and Newton polynomials,
the equalizer theorem plays a role in detecting the 
$\gamma\in \Gamma$ for which there can exist $y\in K^{\times}$ with 
$vy = \gamma$ and $P(y)=0$.

\subsection{Two important pseudo-cauchy sequences.} \label{sec:important pc-sequence} We consider here jammed pc-sequences.
Recall that in $\T$ we have the iterated logarithms  $\ell_n$ with
$$ \ell_0=x, \quad \ell_{n+1}=\log \ell_n,$$
and that this sequence is coinitial in $\T^{>\R}$. By a 
straightforward computation,
$$ \upl_n\ :=\ -\ell_n^{\dagger{}\dagger}\  =\  \frac{1}{\ell_0} + \frac{1}{\ell_0\ell_1} + \cdots + \frac{1}{\ell_0 \ell_1 \cdots \ell_n}.$$
Thus $(\upl_n)$ is a (jammed) pc-sequence in $\T$, but has no pseudolimit 
in $\T$. It does have a pseudolimit in a suitable immediate $H$-field 
extension, and such a limit can be thought of as 
$\sum_{n=0}^\infty  \frac{1}{\ell_0 \ell_1 \cdots \ell_n}$. 

\medskip\noindent
The fact that the
pc-sequence $(\upl_n)$ has no pseudolimit in $\T$ is 
related to a key 
elementary property of $\T$. To explain this we assume in the rest of this 
subsection:\quad  
{\em $K$ is a real closed $H$-field with asymptotic integration}.

\medskip\noindent
To mimick the above iterated logarithms, first take for any $f\succ 1$ in $K$
an $\Log f \succ 1$ in $K$ such that $(\Log f)'\asymp f^\dagger$.
(Think of $\Log f$ as a substitute for $\log f$.) Next, pick a sequence
of elements $\ell_\rho\succ 1$ in $K$, indexed by the ordinals 
$\rho$ less than some infinite limit ordinal: take any
$\ell_0\succ 1$ in $K$, and
set $\ell_{\rho+1}:=\Log(\ell_\rho)$; if $\lambda$ is an infinite limit ordinal
such that all $\ell_\rho$ with $\rho<\lambda$ have been chosen,
then take $\ell_\lambda\succ 1$ in $K$ such that 
$\ell_\rho\succ \ell_\lambda$ for
all $\rho<\lambda$, if there is such an $\ell_\lambda$.  This yields a
sequence $(\ell_\rho)$ with the following properties: \begin{enumerate}
\item[(i)] $\ell_\rho \succ \ell_{\rho'}$ whenever $\rho < \rho'$;
\item[(ii)] $(\ell_\rho)$ is coinitial in $K^{\succ 1}$, that is, 
for each $f\succ 1$ in $K$ there is an index $\rho$ with 
$f\succ \ell_\rho$.
\end{enumerate}
Now set $\upl_{\rho}:= -\ell_{\rho}^{\dagger{}\dagger}$.
One can show that this yields a jammed pc-sequence~$(\upl_{\rho})$ in $K$ and that the pseudolimits of
this pc-sequence in $H$-field extensions of~$K$ do not depend on 
the choice of the sequence $(\ell_\rho)$: different choices yield
``equivalent'' pc-sequences in the sense of \cite{BMS}.
Here is a useful fact about this pc-sequence:

\begin{theorem} If $\upl\in K$ is a pseudolimit of $(\upl_{\rho})$, 
then there is an $H$-field extension $K(\upg)$ such that 
$\upg^\dagger=-\upl$ and $K(\upg)$ has a gap $v(\upg)$. 
\end{theorem}

Every $H$-field has an extension to a real closed
$H$-field with asymptotic integration (for example, a Liouville closure). 
By Theorem~\ref{thimmax} 
we can further arrange that this
extension is maximal, so that all pc-sequences in it have a pseudolimit in it.
Thus the last theorem has the following consequence:

\begin{corollary}\label{gapcor}
Every $H$-field has an $H$-field extension with a gap.
\end{corollary}
 
If $K$ is Liouville closed, then $(\upl_{\rho})$ has no
pseudolimit in $K$. 

\begin{theorem}\label{thwam} The following conditions on $K$ are equivalent: \begin{enumerate}
\item[(1)] $K\models \forall a \exists b\ \big[ v(a-b^\dagger)\  \le vb\  < ({\Gamma^{>}})'\big]$;
\item[(2)] $(\upl_{\rho})$ has no pseudolimit in $K$.
\end{enumerate}
\end{theorem}

Since $\T$ satisfies (2), it also satisfies (1). 
Our discussion preceding Corollary~\ref{gapcor} made it clear that not all 
real closed $H$-fields with asymptotic integration satisfy (2). 
We call attention to the first-order nature of condition (1). 

\medskip\noindent
There is a related and even more important pc-sequence. To define it, set
$$\ome(z)\ :=\   -2z' - z^2 \qquad\text{ for $z\in K$.}$$ 
Then
in $\T$ we have
$$ \upo_n\ :=\ \ome(\upl_n)\ =\  \frac{1}{(\ell_0)^2}+\frac{1}{(\ell_0\ell_1)^2}+
\frac{1}{(\ell_0\ell_1\ell_2)^2} +\cdots+
\frac{1}{(\ell_0\ell_1\cdots\ell_n)^2},$$
so  $(\upo_n)$ is also a jammed
pc-sequence in $\T$ 
without any pseudolimit in $\T$. Likewise, for our real closed $H$-field $K$ 
with asymptotic integration, and setting $\upo_{\rho}:= \ome(\upl_{\rho})$,
the sequence $(\upo_{\rho})$ is a jammed pc-sequence.  
(If $(\upl_{\rho})$ pseudoconverges in $K$, then
so does $(\upo_{\rho})$, but \cite{dagap}
has a Liouville closed example where the converse fails.) 
Here is an analogue of Theorem~\ref{thwam}:

 \begin{theorem}\label{tham} The following conditions on $K$ are equivalent: \begin{enumerate}
\item[(1)] $K\models \forall a \exists b\  \big[ vb < ({\Gamma^{>}})' \text{ and } v(a +\ome(-b^\dagger))\  \le\  2vb  \big]$;
\item[(2)] $(\upo_{\rho})$ has no pseudolimit in $K$;
\item[(3)] $(\upo_{\rho})$ has no pseudolimit in any  
differentially algebraic $H$-field extension of $K$. $($Asymptotic differential transcendence of $(\upo_{\rho})$.$)$
\end{enumerate}
\end{theorem}

The equivalence of (1) and (2) is relatively easy, but to show that (2) implies~(3) is much harder. 
Since $\T$ satisfies (2), it also satisfies (1) and (3). 
The first-order nature of condition (1) will surely play a role in our quest
to characterize the existentially closed $H$-fields by first-order axioms.  
The equivalence of (2) and (3) is related to the following important fact: 

\begin{theorem}\label{coram} Suppose 
$(\upo_{\rho})$ has no 
pseudolimit in $K$. Then $(\upl_{\rho})$ has a
pseudolimit $\upl$ in an immediate $H$-field extension of $K$ 
such that for any pseudolimit $a$ of $(\upl_{\rho})$
in any $H$-field extension of $K$ there is a unique isomorphism 
$K\<\upl\> \to K\<a\>$ over $K$ of ordered valued differential fields sending~$\upl$ to $a$.
\end{theorem}

We define an $H$-field to be {\bf $\upo$-free} if it is real closed, 
admits asymptotic integration, and satisfies the equivalent conditions
of Theorem~\ref{tham}. Any real closed $H$-field that admits asymptotic 
integration and is a directed union of $H$-subfields $F$ for which
$\Psi_F$ has a largest element is $\upo$-free. The property of being 
$\upo$-free is first-order and invariant under compositional conjugation by
positive elements.

\subsection{Simple Newton polynomials.}\label{sec:simple newton pol}
As shown in \cite{vdH}, the Newton polynomials of differential polynomials over $\T$  have the very special form
$$(c_0 + c_1Y + \cdots + c_mY^m)\cdot (Y')^n  \qquad (c_0,\dots, c_m\in \R=C).$$
This fails for some other real closed $H$-fields with asymptotic integration: 

\begin{example*}
Consider the immediate $H$-field extension $K=\R\(( \mathfrak L \)) $ of 
$\mathbb T_{\log}$, where $\mathfrak L=\bigcup_{n=0}^\infty\mathfrak L_n$ 
(see the end of Section~\ref{subsec:transseries}). This $H$-field $K$ admits 
asymptotic integration, and is not $\upo$-free, since it 
contains a pseudolimit 
$\upo\ :=\ \sum_{n=0}^\infty \frac{1}{(\ell_0\ell_1\cdots\ell_n)^2}$ of the
pc-sequence $(\upo_n)$. 
We set
$$P := N - \upo \cdot (Y')^2 \in K\{Y\} \quad\text{where $N(Y) := 2Y'Y''' - 3(Y'')^2 \in \R\{Y\}$.}$$
A somewhat lengthy computation yields  
$N_P=N\notin \mathbb R[Y](Y')^{\N}$.
\end{example*}

It turns out that $\upo$-freeness is exactly what makes
Newton polynomials to have the above simple form:

\begin{theorem} Let $K$ be a real closed $H$-field with asymptotic integration.
Then $K$ is $\upo$-free if and only if the Newton polynomial of any non-zero 
differential polynomial $P\in K\{Y\}$ has the form
$$(c_0 + c_1Y + \dots + c_mY^m)\cdot (Y')^n  \qquad (c_0,\dots, c_m\in C).$$
\end{theorem}

This has a nice consequence for the 
behavior of a differential polynomial near the constant field:

\begin{corollary} Let $K$ be an $\upo$-free $H$-field and 
$P\in K\{Y\}$, $P\ne 0$. Then there
are $\alpha\in \Gamma$, $a\in K^{>C}$ and $m$, $n$ such that
$$ C_L < y < a\ \Longrightarrow\   v(P(y))\ =\ \alpha + mvy + nvy'$$
for all $y$ in all $H$-field extensions $L$ of $K$, where $C_L=\text{constant field of } L$. 
\end{corollary}


We also have the following converse to 
a result from Section~\ref{difhen}:

\begin{corollary} Suppose the $H$-field $K$ is $\upo$-free, and there are
$K$-admissible $\phi>0$ with arbitrarily high $v\phi < (\Gamma^{>})'$ such that 
the flattening of $K^{\phi}$ is differential-henselian. Then 
$K$ is newtonian. 
\end{corollary}

\subsection{Conjectural characterization of existentially closed $H$-fields.}
We can show that every existentially closed 
$H$-field with small derivation is 
$\upo$-free. We already mentioned earlier that
they are Liouville closed, and newtonian, 
and that their linear differential operators 
factor completely after adjoining $\imag=\sqrt{-1}$ to the field. 
Maybe this is the full story:


\medskip\noindent
\begin{optimisticconjecture}
An $H$-field $K$ with small derivation is existentially closed 
if and only if it satisfies the following
first-order conditions:   
\begin{enumerate}
\item[(i)]  $K$ is  Liouville closed;
\item[(ii)] every  $A\in K [\der], A\notin K$, is a product of operators 
of order  $1$ in $ K[\imag][\der]$;
\item[(iii)] $K$ is $\upo$-free;
\item[(iv)] $K$ is newtonian. 
\end{enumerate}
\end{optimisticconjecture}

This conjecture makes the $\T$-Conjecture more precise. It is probably not
optimal as a first-order characterization of  existentially closed $H$-fields.
For example, in the presence of (i), (iii), (iv) we can perhaps
restrict (ii) to $A$ of order $2$. Also, in some arguments
we need the ``newtonian'' property not just for $K$, but also for $K[\imag]$.
We expect the newtonian property of $K[\imag]$ to be a formal consequence of 
$K$ being newtonian,
but if we do not succeed in proving that, we are willing to strengthen (iv)
accordingly. 

It is also conceivable 
that the $H$-field $\T_{\text{log}}$ has a good model theory. It 
satisfies (ii), (iii), (iv), and has some other attractive properties. 
On the other hand, the $H$-field $\T_{\exp}$ of purely exponential 
transseries defines $\mathbb{Z}$; see \cite[Section~13]{AvdD2}.

\section{Quantifier-free Definability}\label{sec:qf definability}

\noindent
In Section~\ref{sec:T-conjecture} we considered three intrinsic model-theoretic statements about~$\T$:

\begin{enumerate} 

\item[(1)]  If $X\subseteq \T^n$ is definable, then $X\cap \R^n$ is semialgebraic.

\item[(2)]  $\T$ is {\it asymptotically o-minimal\/}: for each definable $X\subseteq \T$ there is a $b\in \T$ such that either $(b, +\infty)\subseteq X$ or 
$(b, +\infty)\subseteq \T\setminus X$.

\item[(3)]  $\T$ has NIP.
\end{enumerate}

\noindent
In this section we prove quantifier-free versions of these statements.
First we do this in the easy case when the language
is the natural language $\mathcal L$ of ordered valued differential 
rings. (``Easy'' means here that it follows with very little work 
from results in the literature.) Next we exhibit a basic obstruction\footnote{The term ``obstruction'' is often used to refer to a non-trivial (co)homology
class. Our use here is in the same spirit. In fact, the vanishing of a homology group leads to the
elimination of a quantifier since this vanishing 
means that the existential condition
on a chain $c$ to be a boundary is equivalent to the quantifier-free condition on $c$ that its boundary vanishes.}
showing that $\T$ does not eliminate quantifiers in $\mathcal L$. This obstruction can be lifted by  extending $\mathcal L$ to a language $\mathcal L^*$ which has a unary
function symbol naming a certain integration operator on $\T$. (This
operator is existentially definable in $\T$ using $\mathcal L$.) 
We then show that 
(1), (2), (3) also hold for quantifier-free definable 
relations on $\T$ when the latter is construed as an $\mathcal L^*$-structure.

Thus (1), (2), (3) would follow from the strong form of
the $\T$-Conjecture which says that $\T$
admits quantifier elimination in the language $\mathcal L^*$. This 
form of the  $\T$-Conjecture is unfortunately too strong: In Section~\ref{sec:obstructions} 
we discuss further obstacles, and speculate on how these might be dealt with.

\subsection{Quantifier-free definable sets in $\T$ using $\mathcal L$.} 

Recall:
$$\mathcal L\ =\ \{0,\ 1,\  +,\  -,\ {\cdot}\,\,,\ \der,\ {\le},\ {\preceq}\}.$$
In this subsection we view any $H$-field $K$ as an $\mathcal L$-structure in 
the natural way, and so ``quantifier-free definable'' means 
``definable in $K$ by a 
quantifier-free formula of the language $\mathcal L$ augmented by names for the
elements of $K$.'' The next three propositions contain the
quantifier-free versions of (1)--(3) above.

\begin{prop}\label{1} 
Let $K$ be a real closed $H$-field. If $X\subseteq K^n$ is quantifier-free definable, then 
its trace $X\cap C^n$ in the field $C$ of constants is semialgebraic.
\end{prop}
\begin{proof} Let $P=P(Y_1,\dots, Y_n)\in K\{Y_1,\dots, Y_n\}$ be a 
differential polynomial. Removing from $P$ the terms involving any
$Y_i^{(r)}$ with $r\ge 1$ we obtain an ordinary polynomial 
$p\in K[Y_1,\dots, Y_n]$ such that for all $y_1,\dots, y_n\in C\subseteq K$,
$$P(y_1,\dots, y_n)\ =\ p(y_1,\dots, y_n).$$
Recall also that for all $f,g\in K$ we have 
$$f\preceq g\  \Longleftrightarrow\ \text{$\abs{f}\le c\abs{g}$ for some $c\in C^{>0}$.}$$
It follows that if $X\subseteq K^n$ is quantifier-free definable, then
$X\cap C^n$ is definable (with parameters) in the pair $(K, C)$ construed 
here as the real closed field~$K$ (forgetting its derivation and valuation),
with $C$ as a distinguished subset. This pair $(K, C)$ is a model of 
$\text{RCF}_{\text{tame}}$, as
defined in \cite{vdD}. By Proposition~8.1 of \cite{vdD} applied to 
$T=\text{RCF}$, a subset of $C^n$ which is definable (with parameters) 
in the pair $(K, C)$ is semialgebraic in the sense of $C$.
\end{proof}

We now turn to
quantifier-free asymptotic o-minimality. This follows easily from 
the {\em logarithmic decomposition\/} of a differential 
polynomial in \cite{vdH}, as we explain now. Let $K$ be a differential field. 
For $y\in K$, we set $y^{\< 0\>}:= y$, and
inductively, if $y^{\<n\>}\in K$ is defined and non-zero, 
$y^{\<n+1\>}:=  (y^{\<n\>})^{\dagger}$ (and otherwise $y^{\<n+1\>}$ is not defined).
Thus in the differential fraction field $K\langle Y \rangle$ of the 
differential polynomial ring $K\{Y\}$
each $Y^{\<n\>}$ is
defined, the elements $Y^{\<0\>}, Y^{\<1\>}, Y^{\<2\>},\dots$
are algebraically independent over $K$, and 
$$K\langle Y \rangle = K\big(Y^{\<n\>}:\ n=0,1,2,\dots\big).$$ 
If $y^{\<n\>}$ is defined and $\i=(i_0,\dots, i_n)\in \mathbb{N}^{1+n}$, 
we set $$  y^{\<\i\>}\ :=\ 
(y^{\<0\>})^{i_0}(y^{\<1\>})^{i_1}\cdots (y^{\<n\>})^{i_n}. $$
One can show that any $P\in K\{Y\}$ of order $\le r$ has a unique decomposition
\[\begin{array}{lll} P\ =\ &\sum_{\i}P_{\<\i\>}Y^{\<\i\>}&  
\text{{\bf (logarithmic decomposition)},}
\end{array}\]
with $\i$ ranging over $\N^{1+r}$, all $P_{\<\i\>}\in K$, and
$P_{\<\i\>}\ne 0$ for only finitely many $\i$.

Consider the case $y\in K:=\T$. Then $y^{\<1\>}=y^{\dagger}$ is defined 
for $y\ne 0$, and if $y>\exp(x^2)$, then $y^{\<1\>}> 2x$ and $y> (y^{\<1\>})^m$
for all $m$. By induction on $n$, if $y>\exp^{n+1}(x^2)$, 
with the exponent $n+1$ referring to compositional iteration, then
$y^{\<n+1\>}$ is defined, $y^{\<n\>}> \exp^{n}(x^2)$, and $y^{\<n+1\>} >(y^{\<n\>})^m$
for all $m$. 

Let a non-zero $P\in \T\{Y\}$ of order $\le r$ be given with
the logarithmic decomposition displayed before. Take $\j\in \N^{1+r}$
lexicographically maximal with
$P_{\<\j\>}\ne 0$.  It follows from the above that
we can take $b\in \T$ so large that if
$y>b$, then $y^{\<r\>}$ is defined and
$P(y) \sim\ P_{\<\j\>}y^{\<\j\>}$ (where $f\sim g$ means $f-g\prec g$). In particular, if $P_{\<\j\>}>0$,
then $P(y)>0$ for all $y>b$, and if $P_{\<\j\>}<0$,
then $P(y)<0$ for all $y>b$. By similar reasoning, given any non-zero 
$P, Q\in \T\{Y\}$, there is $b\in \T$ such that either 
$P(y) \preceq Q(y)$ for all $y>b$ in $\T$, or
$P(y) \succ Q(y)$ for all $y>b$ in $\T$. Thus:

\begin{prop}\label{2} 
If $X\subseteq \T$ is quantifier-free definable, then there 
is $b\in \T$ such that either $(b, +\infty)\subseteq X$, or  
$(b, +\infty)\subseteq \T\setminus X$.
\end{prop}

This proposition holds for any Liouville closed $H$-field $K$ instead of
$\T$: we can define on such $K$ a substitute for the exponential function
$\exp$ as used in the proof above, see \cite[Section~1.1]{AvdD2}.

\medskip\noindent
A relation $R\subseteq A\times B$ is said to be {\bf independent} if
for every $N\ge 1$ there are elements $a_1,\dots, a_N\in A$ and $b_I\in B$, for 
each $I\subseteq \{1,\dots, N\}$, such that
$$R(a_i, b_I)\ \Longleftrightarrow\ i\in I \qquad (\text{for }i=1,\dots,N,\ \text{ and all } I\subseteq \{1,\dots, N\}).$$ 
A (one-sorted) structure 
$\boldsymbol{M}=(M; \dots)$ is said to have {\bf NIP} (the {\bf Non-Independence Property})
if there is no independent definable relation $R\subseteq M^m\times M^n$.
This is a robust model-theoretic tameness condition on a structure. 
It was introduced early on 
by Shelah~\cite{S1}; there is also a substantial body of recent work
around this notion, see for example~\cite{HPP}.
Stable structures as well as o-minimal structures have NIP.

\begin{prop}\label{3}
Let $K$ be an $H$-field.
No quantifier-free definable relation $R\subseteq K^m\times K^n$ is
independent.
\end{prop}
\begin{proof} Let $\operatorname{OVDF}$ be the $\mathcal L$-theory of ordered, valued, 
differential fields where the only axiom relating the ordering, valuation, 
and derivation is 
$$\forall x \forall y\ (0\le x\le y\ \rightarrow\ x\preceq y).$$ 
Guzy and Point~\cite[Corollary~6.4]{GP} 
show that  $\operatorname{OVDF}$ has a model completion $\operatorname{OVDF}^c$,
and that $\operatorname{OVDF}^c$ has NIP. Now use an embedding of $K$ into some
model of $\operatorname{OVDF}^c$.
\end{proof}

\subsection{$\T$ does not admit quantifier elimination in $\mathcal L$.} Let 
$K$ be an $H$-field. Then we have the $\mathcal{O}$-submodule
$$ \I(K):=\ \{y\in K:\text{$y\preceq f'$ for some $f\in \smallo$}\}$$
of $K$, with $\der \smallo\subseteq \I(K)$. If the derivation $\der$ of $K$ is small, then $\I(K)$ is an ideal in $\mathcal{O}$. If $K$ is Liouville closed, then
$$  \der \smallo\ =\ \I(K)\ =\ 
\{y\in K:\text{$y \prec f^\dagger$ for all non-zero $f\in \smallo$}\},$$
so $\I(K)$ is existentially as well as universally definable in the 
$\mathcal L$-structure $K$.
Still considering $\T$ as an $\mathcal L$-structure, we have:

\begin{prop}\label{nqel} The subset $\I(\T)$ of $\T$ is not 
quantifier-free definable in $\T$.
\end{prop} 
\begin{proof}
Recall from Section~\ref{sec:important pc-sequence} the pc-sequence 
$(\upl_n)$ in $\T$:
$$\upl_n=- \ell_n^{\dagger{}\dagger}\  =\ -\left(\frac{1}{\ell_n}\right)^{\dagger{}\dagger}\ =\ \frac{1}{\ell_0} + \frac{1}{\ell_0\ell_1} + \cdots + \frac{1}{\ell_0 \ell_1 \cdots \ell_n}.$$
It has no pseudolimit in $\T$. Fix some $\aleph_1$-saturated elementary extension $K$ of $\T$ and take
$\ell\in K$ such that $\ell > C$ but $\ell< \ell_n$ for all $n$.
Then $\upl:=-\ell^{\dagger{}\dagger}=-(1/\ell)^{\dagger{}\dagger}$ 
is a pseudolimit of $(\upl_n)$. An easy computation gives
$$-(1/\ell_n)'^{\dagger}\ =\ \upl_n + (1/\ell_0\cdots \ell_n),$$
so $\big(-(1/\ell_n)'^{\dagger}\big)$ is a pc-sequence with 
the same pseudolimits in $K$ as $(\upl_n)$.
Now $a:= -(1/\ell)'^{\dagger}$ is a pseudolimit of
$\big(-(1/\ell_n)'^{\dagger}\big)$, so by Theorem~\ref{coram}, the  
$H$-subfields $\T\langle \upl \rangle$ and $\T\langle a \rangle$ of $K$
are immediate extensions of $\T$, and 
we have an isomorphism $\T\langle \upl \rangle \to \T\langle a \rangle$
over $\T$ that sends $\upl$ to $a$. 
The element $f=(1/\ell)^\dagger$ of $K$ satisfies
$f^\dagger = -\upl$ and $\phi'\prec f\prec \phi^\dagger$ for all $\phi\in\T^{\times}$ with $\phi\prec 1$, and the real closure
$\T\langle \upl \rangle^{\operatorname{rc}}$ of $\T\<\upl\>$ in $K$ is an immediate extension of $\T$.
Hence in the terminology of \cite[Section~12]{AvdD2} and using 
\cite[Proposition~12.4]{AvdD2},~$-\upl$ creates a gap over $\T\langle \upl \rangle^{\operatorname{rc}}$. 
Since $g=(1/\ell)'$ satisfies $g^\dagger=-a$, the above isomorphism $\T\<\upl\} \to \T\<a\>$ extends by \cite[Lemma~12.3]{AvdD2} and the uniqueness statement in
\cite[Lemma~5.3]{AvdD1} to an isomorphism
$$  \T\langle \upl, f \rangle \to \T\langle a, g \rangle$$
of $\mathcal L$-structures which sends $f$ to $g$. Now, if $\I(\T)$ 
were defined 
in $\T$ by a quantifier-free formula $\phi(y)$ 
in the language $\mathcal L$ augmented by names for the elements of $\T$, then we
would have $K \models \neg \phi(f)$ and $K\models \phi(g)$, and so
$\T\langle \upl, f\rangle \models \neg \phi(f)$ and 
$\T\langle a,g \rangle\models \phi(g)$, which violates the above
isomorphism between  $\T\langle \upl, f\rangle$ and  $\T\langle a,g \rangle$.
\end{proof}

For later use it is convenient to extend the language
$\mathcal L$ as follows. Note that~$\mathcal L$ has the language of ordered rings as a sublanguage. We consider $\R$ as a structure for the language of ordered rings in the usual way. A function $\R^n \to \R$ is said to be {\bf $\Q$-semialgebraic\/} if its graph is defined in the structure~$\R$ by a (quantifier-free) formula in the language
of ordered rings; we do not allow names for arbitrary real numbers
in the defining formula. We extend $\mathcal L$ to the language $\mathcal L'$ by adding for each $\Q$-semialgebraic function 
$f\colon \R^n \to \R$ an $n$-ary function symbol $f$. We construe any real closed valued differential field $K$ as an $\mathcal L'$-structure
by associating to any $\Q$-semialgebraic function $f\colon \R^n\to \R$ the function $K^n \to K$ whose graph is defined in $K$ by any formula
in the language of ordered rings that defines the graph of $f$ in $\R$.  

For example, the function $y\mapsto y^{-1}\colon \T \to \T$, with $0^{-1}:= 0$ by convention, is named by a function symbol of $\mathcal L'$, and so is, for each integer $d\ge 1$, the function
$y\mapsto y^{1/d}\colon \T \to \T$, taking the value $0$ for $y\le 0$ by convention. 

\begin{prop}\label{qfqf} If $X\subseteq \T^n$ is 
quantifier-free definable in $\T$ as $\mathcal L'$-structure, 
then $X$ is quantifier-free definable in
$\T$ as $\mathcal L$-structure.
\end{prop}

One can view this as a partial quantifier elimination: it is obvious
how to eliminate occurrences
of function symbols of $\mathcal L'\setminus \mathcal L$  
from a quantifier-free $\mathcal L'$-formula 
at the cost of introducing existentially quantified new variables, 
and Proposition~\ref{qfqf} says that we can eliminate those 
quantifiers again {\em without reintroducing these function symbols}. 
This fact can be proved by explicit
means, but we prefer a model-theoretic argument that we can use
also in later situations where explicit elimination would be very tedious.

To formulate this in sufficient generality, let $L$ be a sublanguage of
the (one-sorted) first-order language ${L}^*$, and assume that
$L$ has a constant symbol. Let ${\boldsymbol A}^*=(A;\dots)$ and ${\boldsymbol B}^*$
range over ${L}^*$-structures, and let ${\boldsymbol A}$ and ${\boldsymbol B}$
be their $L$-reducts. Let $T^*$ be an ${L}^*$-theory.   
Then we have the following criterion:

\begin{lemma} Let $\varphi^*(x)$ with $x=(x_1,\dots,x_n)$ be an 
${L}^*$-formula. Then $\varphi^*(x)$ is $T^*$-equivalent to some 
quantifier-free ${L}$-formula $\varphi(x)$ iff for all ${\boldsymbol A}^*,{\boldsymbol B}^*\models T^*$,
common $L$-sub\-struc\-tures ${\boldsymbol C}=(C;\dots)$ of ${\boldsymbol A}$ and
${\boldsymbol B}$, and~$c\in C^n$:
$${\boldsymbol A}^*\models \varphi^*(c) \quad\Longleftrightarrow\quad {\boldsymbol B}^*\models \varphi^*(c).$$
\end{lemma}

This criterion is well-known (at least for ${L}={L}^*$), and follows
by a standard model-theoretic compactness argument. Typically, 
the criterion gets used via its corollary below. To state that corollary, we
define $T^*$ to have
{\bf closures of $L$-substructures} if 
for all ${\boldsymbol A}^*,{\boldsymbol B}^*\models T^*$ with a common 
${L}$-substructure $\boldsymbol C=(C;\dots)$ of
${\boldsymbol A}$ and ${\boldsymbol B}$, there is a (necessarily unique) 
isomorphism from the $L^*$-substructure of ${\boldsymbol A}^*$ generated by
$C$ onto the $L^*$-substructure of ${\boldsymbol B}^*$ generated by
$C$ which is the identity on $C$. 

\begin{cor}\label{closureprop}
If $T^*$ has closures of $L$-substructures, then every
quan\-ti\-fier-free ${L}^*$-formula is $T^*$-equivalent to a quantifier-free
$L$-formula.
\end{cor}

\begin{proof}[Proof of Proposition~\ref{qfqf}]
We are going to apply Corollary~\ref{closureprop} with 
\begin{align*}
L\ :&=\ \mathcal L, \qquad  L^* :=\ \mathcal L',\\
T^* :&= \text{ the $L^*$-theory of real closed valued differential fields}.
\end{align*}
Indeed, we show that $T^*$ has 
closures of $L$-substructures. Let $E,F\models T^*$ have a common 
$L$-substructure $D$. Thus $D$ is an ordered differential subring of both $E$ and $F$ such that
for all $f,g\in D$ we have $f\preceq_E g\Longleftrightarrow f\preceq_D g\Longleftrightarrow f\preceq_F g$, where $\preceq_D$, $\preceq_E$, $\preceq_F$ are the
interpretations of the symbol $\preceq$ of $L$ in $D$, $E$, $F$, respectively.
Let $K_E$ and $K_F$ be the fraction fields of the integral domain~$D$ in~$E$
and $F$ respectively. Then
$K_E$ is the underlying ring of an $L$-substructure of $E$, to be denoted also 
by $K_E$. Likewise,
$K_F$ denotes the corresponding $L$-substructure of $F$, and we have a unique
$L$-isomorphism $K_E \to K_F$ that is the identity on $D$.
Let $K_E^{\operatorname{rc}}$ and $K_F^{\operatorname{rc}}$ be the real closures of 
the ordered fields~$K_E$ and $K_F$ in $E$ and $F$, respectively. Then
$K_E^{\operatorname{rc}}$ is the underlying ring of an $L^*$-substructure of $E$, to be denoted also by $K_E^{\operatorname{rc}}$. Likewise,
$K_F^{\operatorname{rc}}$ denotes the corresponding $L^*$-substructure of $F$, 
and the above $L$-isomorphism $K_E \to K_F$ extends uniquely to
an $L^*$-isomorphism $K_E^{\operatorname{rc}} \to K_F^{\operatorname{rc}}$. It 
remains to note that~$K_E^{\operatorname{rc}}$ is the $L^*$-substructure of $E$
generated by $D$.
\end{proof}

Let $K$ be a real closed valued differential field. Then a set $X\subseteq K^n$ 
is said to be {\bf $\Q$-semialgebraic} if it is defined in $K$ by
some (quantifier-free) formula in the language of ordered rings, 
and a function $K^n \to K$ is said to be
{\bf $\Q$-semialgebraic} if its graph is.

\subsection{Adding a new primitive.} Let $\mathcal L'_{\I}$
be the language $\mathcal L'$ augmented by a unary
predicate symbol $\I$. We construe $\T$ as an $\mathcal L'_{\I}$-structure by 
interpreting~$\I$ as $\I(\T)$. In view of Proposition~\ref{nqel} and 
Lemma~\ref{qfqf} 
this genuinely changes what can be defined quantifier-free in $\T$.
Nevertheless, Propositions~\ref{1},~\ref{2},~\ref{3} (in the case $K=\T$) go through 
when ``quantifier-free'' is with respect to $\T$ as $\mathcal L'_{\I}$-structure. 
For ``quantifier-free NIP'' we can almost repeat the previous argument:

\begin{prop}\label{prop:NIP} No quantifier-free definable relation 
$R\subseteq \T^{m}\times \T^n$ on $\T$ as an $\mathcal L'_{\I}$-structure is independent.
\end{prop}
\begin{proof} The set $\I(\T)$ is convex in $\T$. Embedding the $\mathcal L$-structure 
$\T$ in a sufficiently saturated model $\boldsymbol{M}$ of $\operatorname{OVDF}^c$, we
can take $a>0$ in  $\boldsymbol{M}$ such that $(-a,a)\cap \T=\I(\T)$, where the 
interval $(-a,a)$ is with respect to $\boldsymbol{M}$. Now use that $\boldsymbol{M}$
has NIP.
\end{proof}

Let $K$ be a real closed $H$-field, and $r\in \N$. For $y=(y_1,\dots, y_m)\in K^m$ we set $y':=(y_1',\dots, y_m')\in K^m$, and accordingly we define
$$\big(y, y',\dots, y^{(r)}\big):= \big(y_1,\dots, y_m, y_1',\dots, y_m',\dots, y_1^{(r)},\dots, y_m^{(r)}\big)\in K^{m(1+r)}.$$
A {\bf $\der$-covering} (of order $r$) of a function $g\colon K^m \to K$ consists of a finite covering $\Ca$ of $K^{m(1+r)}$ by $\Q$-semialgebraic sets and for each $S\in \Ca$ a $\Q$-semialgebraic function $g_S\colon K^{m(1+r)} \to K$ such that
$$g(y) = g_S\big(y, y',\dots, y^{(r)}\big)\ \text{  for all $y\in K^m$ with $\big(y, y',\dots, y^{(r)}\big)\in S$.}$$ 
For example, if $P\in \Q\{Y_1,\dots, Y_m\}$ is a differential polynomial of
order $\le r$, then the function $y \mapsto P(y)\colon K^m \to K$ has a $\der$-covering of order $r$ consisting just of a single
set, namely $K^{m(1+r)}$. It is easy to see that if $f\colon K^n \to K$ is 
$\Q$-semialgebraic and $g_1,\dots, g_n\colon K^m \to K$ have $\der$-coverings (of various orders), then
$f(g_1,\dots, g_n)\colon K^m \to K$ has a $\der$-covering. In particular, the sum $g_1+g_2$ of functions $g_1, g_2\colon K^m \to K$ with
$\der$-coverings has a $\der$-covering, and so does their product $g_1g_2$. 
Less obviously:

\begin{lemma} If $g\colon K^m \to K$ has a $\der$-covering, then so does the
function $$y \mapsto g(y)'\colon  K^m \to K.$$
\end{lemma}
\begin{proof} Let $\Ca$ be a $\der$-covering of $g$ of order $r$ with, for each
set $S\in \Ca$, the witnessing function $g_S\colon K^{m(1+r)} \to K$. By further partitioning we can arrange that each set $S\in \Ca$ is a $\Q$-semialgebraic cell
which is of class $C^1$ in the sense that the standard projection map $p_S\colon S \to  p(S)$
onto an open cell $p(S)\subseteq K^d$, with $d=\dim S$, is not just a homeomorphism, but even a
diffeomorphism of class $C^1$ (in the sense of the real closed field $K$).
In addition we can arrange that for each witnessing map $g_S\colon K^{m(1+r)} \to K$
the restriction of $g_S$ to $S$ is of class $C^1$. Let us now focus on one particular $S\in \Ca$, and first consider the case that $S$ is open in $K^{m(1+r)}$. 
Then by the $\Q$-semialgebraic version of Lemma~4.4 in \cite{AvdD2}, and Remark (2) following its proof, we have $\Q$-semialgebraic functions
$h, h_1,\dots, h_{mr}\colon K^{m(1+r)} \to K$, such that for all
$\vec y = (y_{10},\dots, y_{m0},\dots, y_{1r},\dots,y_{mr})\in S$, 
$$ g_S(\vec y)'\  =\ h(\vec y) + \sum_{i=1}^m\sum_{j=0}^r h_{ij}(\vec y)y_{ij}'.$$
If $S$ is not open, of dimension $d$, this statement remains true,
as one can see by reducing to the case of the open cell $p(S)\subseteq K^d$
via the $C^1$-diffeomorphism $p_S\colon S \to p(S)$.
It follows easily that the function $y\mapsto g(y)'$ has a $\der$-covering of
order $r+1$, whose sets are the products $S\times K^m$ with $S\in \Ca$. 
\end{proof}

It follows that if $t(y_1,\dots, y_n)$ is an $\mathcal L'_{\I}$-term, then the function $$b\mapsto t(b)\colon K^n \to K$$ 
has a $\der$-covering. This is now used to prove:

\begin{prop}\label{indconst} If $X\subseteq \T^n$ is quantifier-free definable in $\T$ as an
$\mathcal L'_{\I}$-structure, then $X\cap \R^n$ is semialgebraic.
\end{prop}
\begin{proof} By the remark preceding the proposition, and the arguments
in the proof of Proposition~\ref{1}, it suffices to show the following. Let
$s(x,y)$ and~$t(x,y)$ with $x=(x_1,\dots, x_m)$ and $y=(y_1,\dots,y_n)$ be $\mathcal L'_{\I}$-terms in which
the function symbol $\der$ does not occur, and let $a\in \T^m$. Then the sets 
\begin{align*}
&\big\{b\in \R^n:\ t(a,b)=0\big\}, &  &\big\{b\in \R^n:\ t(a,b)>0\big\}, \\ 
&\big\{b\in \R^n:\ s(a,b) \preceq t(a,b)\big\},&  &\big\{b\in \R^n:\ t(a,b)\in \I(\T)\big\}
\end{align*}
are semialgebraic subsets of $\R^n$.  Since 
the function $b\mapsto t(a,b)\colon \T^n \to \T$ is
semialgebraic in the sense of $\T$, this holds for the first three sets
by the argument 
at the end of the proof of Proposition~\ref{1}. For the last set, take 
some real closed 
field extension $K$ of $\T$ with a positive element $c$ such that
$\I(\T)=\T\cap (-c,c)$, where the interval $(-c,c)$ is in the sense of $K$.
Then 
$$\big\{b\in \R^n: t(a,b)\in \I(\T)\big\}=\big\{b\in \R^n: \abs{t(a,b)}< c\big\},$$ 
which is the trace in $\R$ of a semialgebraic subset of $K^n$. Such traces are known to be semialgebraic in the sense of $\R$.
\end{proof}

In proving next that $\T$ qua
$\mathcal L'_{\I}$-structure
is quantifier-free asymptotically o-minimal, we shall use the easily 
verified fact that if $K$ is an $H$-field that admits asymptotic 
integration, and $L$ is an $H$-field extension
of $K$, then $\I(L)\cap K=\I(K)$.

\begin{prop}\label{prop:2 analogue}
If $X\subseteq\T$ is quantifier-free definable in $\T$ as $\mathcal L'_{\I}$-structure, then for some $f\in\T$, either $(f,+\infty)\subseteq X$ or $(f,+\infty)\subseteq\T\setminus X$.
\end{prop}

\begin{proof}
Let $K$ be an elementary extension of $\T$ as $\mathcal L'_{\I}$-structure and $a,b\in K$, $a> \T$, $b>\T$. By familiar model-theoretic arguments it suffices to show that 
then there is an isomorphism of $\mathcal L$-structures
$\T\<a\>^{\operatorname{rc}}\to \T\<b\>^{\operatorname{rc}}$ over $\T$ that
sends $a$ to $b$, and maps $\I(K)\cap\T\<a\>^{\operatorname{rc}}$ onto $\I(K)\cap\T\<b\>^{\operatorname{rc}}$.
(Here $\operatorname{rc}$ refers to the real closure in $K$.) 
Proposition~\ref{2} gives an isomorphism of $\mathcal L$-structures $\T\<a\>\to \T\<b\>$ over $\T$ sending $a$ to $b$, and this isomorphism extends uniquely to an $\mathcal L$-isomorphism 
$\T\<a\>^{\operatorname{rc}}\to \T\<b\>^{\operatorname{rc}}$. 
The arguments preceding Proposition~\ref{2} show that for all $m,n\ge 1$,
$$v(a^{\<n-1\>})\ <\ mv(a^{\<n\>})\ <\ v(\T^\times),$$
so $a$ is differentially transcendental over $\T$, and the value group of $\T\<a\>^{\operatorname{rc}}$ is 
$$v(\T^\times)\oplus\bigoplus_n \Q v(a^{\<n\>})\qquad\text{(internal direct sum of $\Q$-subspaces)},$$
which contains $v(\T^\times)$  
as a convex subgroup. It also follows that $\T\<a\>^{\operatorname{rc}}$  
is an $H$-field, with the same constant field $\R$ as $\T$.
Therefore, $\T\<a\>^{\operatorname{rc}}$ admits asymptotic integration, so
$\I(K)\cap\T\<a\>^{\operatorname{rc}} = \I(\T\<a\>^{\operatorname{rc}})$. 
Likewise, $\I(K)\cap\T\<b\>^{\operatorname{rc}} = \I(\T\<b\>^{\operatorname{rc}})$, 
hence our isomorphism 
$\T\<a\>^{\operatorname{rc}}\to \T\<b\>^{\operatorname{rc}}$  maps $\I(K)\cap\T\<a\>^{\operatorname{rc}}$ onto $\I(K)\cap\T\<b\>^{\operatorname{rc}}$ as required.
\end{proof}

This result tells us how a quantifier-free definable $X\subseteq
\T$ behaves near~$+\infty$. Using
fractional linear transformations we get analogous 
behavior to the left as well as to the right of any point in $\T$.
In other words, the $\mathcal L'_{\I}$-structure~$\T$ is 
quantifier-free locally o-minimal. (In this connection we note that 
local o-minimality by itself does not imply NIP~\cite[Example~6.19]{Fornasiero}.)

\subsection{Expanding by small integration.} Next we show that  ``small
integration'' can be eliminated from quantifier-free formulas. This is
a further partial quantifier elimination in the style of 
Proposition~\ref{qfqf}.
  
Let $K$ be an $H$-field.  
We have $\der\smallo\subseteq \I(K)$, and we say that $K$ 
{\bf admits small integration} if $\der\smallo = \I(K)$. Liouville
closed $H$-fields admit small integration.
It follows from Section~3 and Proposition~4.3 in \cite{AvdD1} 
that $K$ has an immediate $H$-field extension $\csi(K)$ that is
henselian as a valued field and admits small integration, with the 
following universal property:  for any $H$-field extension~$L$ of $K$ 
that is henselian as a valued field and admits small integration 
there is a unique $K$-embedding of $\csi(K)$ into $L$. We call
$\csi(K)$ the {\bf closure of $K$ under small integration}.

\medskip\noindent
Let $K$ be an $H$-field admitting small integration. 
The derivation $\der$ is injective on $\smallo$, so we can define 
$\int\colon K \to K$ by 
$$\int a' = a\ \text{ for $a\in \smallo$,} \qquad \int b = 0\ \text{ for $b\notin \der\smallo$.}$$
Note that the standard part map
$\st\colon K \to K$ defined by 
$$\st(c+\epsilon)=c \text{ for $c\in C$, $\epsilon \prec 1$,}\qquad \st(a)=a \text{ for $a\succ 1$,}$$ 
can be expressed in terms of $\int$ by $ \st(a) = a-\int a'$. The reason for mentioning this fact is that such a standard part map is used to eliminate quantifiers in certain expansions of o-minimal fields; see \cite[(5.9)]{vdDL}.

\medskip\noindent
Real closed $H$-fields admitting small
integration are construed below as $\mathcal L^*$-structures 
where $\mathcal L^*$ is a language extending $\mathcal L'_{\I}$ 
by a new unary function symbol $\int$, to be 
interpreted as indicated above.

\medskip\noindent
Let $T^*$ be the $\mathcal L^*$-theory of real closed $H$-fields admitting small integration. Then we have the following elimination result:

\begin{prop}\label{qint}
$T^*$ has closures of $\mathcal L'_{\I}$-substructures. Thus
every quanti\-fier-free $\mathcal L^*$-formula
is $T^*$-equivalent to a quantifier-free 
$\mathcal L'_{\I}$-formula.
\end{prop}
\begin{proof}
Let $K$ be a model of $T^*$ and let $E$ be an $\mathcal L'_{\I}$-substructure of $K$.
Then $E$ is a real closed pre-$H$-field, and
we may  consider the $H$-field closure $H(E)$ of $E$ as
an $H$-subfield of $K$, with real closure $H(E)^{\operatorname{rc}}$ in $K$. We let $E^*:=\csi\big(H(E)^{\operatorname{rc}}\big)$ be the closure under small integration of $H(E)^{\operatorname{rc}}$, viewed as an $H$-subfield of $K$. In fact, $E^*$ is real closed and closed under small integration, hence
an $\mathcal L^*$-substructure of $K$. 
Let $L$ be another model of $T^*$ containing $E$ as 
$\mathcal L'_{\I}$-substructure; we need to show that the 
natural inclusion $E\to L$ extends to an embedding of 
$\mathcal L^*$-structures $E^*\to L$.
By the universal properties of $H$-field closure, real closure, and 
closure under small integration, there is an embedding of 
$\mathcal L'$-structures $E^*\to L$ which extends the 
inclusion $E\to L$.
This embedding also preserves the interpretations of the symbol $\int$ in $K$ respectively $L$; so after identifying $E^*$ with its image under this embedding, it remains to show that $\I(K)\cap E^*=\I(L)\cap E^*$.
For this we distinguish two cases:

\subsubsection*{Case 1: there is $r\in \mathcal O_E\setminus C_E$ with $v(r')\notin (\Gamma_E^{>})'$.}
Take such $r$, and take $y\in H(E)$ with $y'=r'$ and $\alpha:= v(y)>0$. 
Then by \cite[Corollary~4.5,~(1)]{AvdD1},
$\Gamma_{H(E)} =\Gamma_E \oplus \mathbb Z\alpha$ with $0<n\alpha < \Gamma_E^{>}$
for all $n\ge 1$. Also, $\max \Psi_{H(E)}=\alpha^\dagger$.
It follows easily that 
$$\I(K)\cap H(E)\ =\ \big\{f\in H(E): vf > \alpha^\dagger\big\}\ =\ \I(L)\cap H(E).$$
This remains true when we replace $H(E)$ by $E^*$, since 
$$\Gamma_{E^*} =\ \text{divisible hull of $\Gamma_{H(E)}$}\ =\ \Gamma_E\oplus\Q\alpha,$$ 
and so $\max\Psi_{E^*}=\alpha^\dagger$.

\subsubsection*{Case 2: there is no such $r$.}
Then by  \cite[Corollary~4.5,~(2)]{AvdD1} we have $\Gamma_{H(E)}=\Gamma_E$, and hence
$\Gamma_{E^*}=\Gamma_E$. Now $\I(K)\cap E=\I(L)\cap E$ gives
$(\Gamma_K^{>})'\cap\Gamma_E = (\Gamma_L^{>})'\cap\Gamma_E$, so
	$ (\Gamma_K^{>})'\cap\Gamma_{E^*} = (\Gamma_L^{>})'\cap\Gamma_{E^*}$, and thus
$\I(K)\cap E^*=\I(L)\cap E^*$.
\end{proof}

We can do the same for {\em small exponentiation}: 
given $a\prec 1$ in $\T$, its exponential~$\ex^a$ is the unique
element $1+y$ with $y\prec 1$ in $\T$ such that $y'=(1+y)a'$. Thus
the bijection $a\mapsto \ex^ a\colon \smallo \to 1+\smallo$ is (existentially and
universally) definable in the $\mathcal{L}$-structure $\T$.
Arguments as in the proof of Proposition~\ref{qint} show that
expanding the $\mathcal L^*$-structure $\T$ 
by this operation (taking the value~$0$ on $\T\setminus \smallo$, by convention)
does not change what is quantifier-free definable.

\section{Further Obstructions to Quantifier Elimination}\label{sec:obstructions}

\noindent
The language $\mathcal{L}'_{\I}$ is 
rather strong as to 
what it can express quantifier-free about $\T$, as we have seen. However, 
$\T$ does not admit QE 
in this language. To discuss this, let $K$ be a Liouville closed 
$H$-field, and consider the
subset
$$ \Upl(K)\ :=\  -(K^{>C})^{\dagger\dagger}\ =\  \{-a^{\dagger{}\dagger}:\  
a\in K,\ a>C\}$$
of $K$. In $\T$ the sequence $(\upl_n)$ given by 
$$\upl_n\ :=\ -\ell_n^{\dagger{}\dagger}\ =\
\frac{1}{\ell_0}+\frac{1}{\ell_0\ell_1}+\cdots+
\frac{1}{\ell_0\ell_1\cdots\ell_n}$$ is cofinal in $\Uplambda(\T)$. 
As with most of this paper we omit proofs for what we claim below: 
these proofs are either straightforward, or very similar to proofs of 
analogous results in Section~\ref{sec:qf definability}, or would require many extra pages.    

\begin{lemma}
Let $K$ be a Liouville closed $H$-field. 
Then $\Upl(K)$ is closed downward: if
$f\in K$ and $f<g\in \Upl(K)$, then $f\in \Upl(K)$. For $f\in K^{\times}$ we have
$$  f\in \I(K)\ \Longleftrightarrow\ -f^\dagger \notin \Upl(K).$$
\end{lemma}

This follows easily from results in \cite{AvdD1} and \cite{AvdD2}. 
In particular, $\I(\T)$ is quantifier-free definable from $\Upl(\T)$
in the $\mathcal L'$-structure $\T$.  
A refinement of the proof of Proposition~\ref{nqel} shows that we cannot
reverse here the roles of $\I(\T)$ and $\Upl(\T)$:

\begin{lemma} $\Upl(\T)$ is not quantifier-free definable in the 
$\mathcal{L}'_{\I}$-structure $\T$.
\end{lemma}

Let $\mathcal L'_{\Upl}$ be the language
$\mathcal L'$ augmented by a unary predicate symbol $\Upl$, 
to be interpreted in $\T$ as $\Upl(\T)$. What we proved in Section~\ref{sec:qf definability} for $\T$ as
$\mathcal L'_{\I}$-structure goes through for $\T$ as
$\mathcal L'_{\Upl}$-structure. However, we run into a new
obstruction involving the function $\ome(z)=-z^2-2z'$. To explain this, we
first summarize some basic facts about this function $\ome$ on $\T$. (See Figure~\ref{fig:omega} for a sketch of $\ome$.)

\begin{lemma} The restriction of $\ome\colon \T \to \T$ to $\Upl(\T)$ is 
strictly increasing and has the intermediate value property. Also, 
$\ome(\T)=\ome(\Upl(\T))$, and thus the sequence $(\upo_n)$ with 
$\upo_n:=\ome(\upl_n)$ is strictly increasing and 
cofinal in $\ome(\T)$.
\end{lemma}

\begin{figure}
\begin{tikzpicture}[scale=0.74]
\draw (4.5,0) -- (8.5,0);
\draw (9.5,0) -- (12,0); 
\draw[densely dotted] (0,2) -- (12,2);
\draw[densely dotted] (-5,2) -- (-3.75,2);
\draw[densely dotted] (4,7) -- (4,0);
\draw[densely dotted] (4,-1) -- (4,-5);
\draw[densely dotted] (9,7) -- (9,0);
\draw[densely dotted] (9,-1) -- (9,-5);
\draw[densely dotted] (-4,7) -- (-4,0);
\draw[densely dotted] (-4,-1) -- (-4,-5);

\draw (0,1.5) -- (0,-3);
\draw (0,-4.5) -- (0,-5);
\draw[densely dotted] (0,1.5) -- (0,1.75);
\draw[densely dotted] (0,2.25) -- (0,2.5);

\draw[densely dotted] (3.5,0) -- (3.75,0);
\draw[densely dotted] (4.25,0) -- (4.5,0);
\draw[dotted] (3.75,0) -- (4.25,0);
\draw (4.0,-0.5) node {\mbox{\fontsize{9}{0}\selectfont $\frac{1}{\ell_0\ell_1\cdots}$}};
\fill[white] (4,0) circle (0.2em);
\draw (4,0) circle (0.2em);

\draw (-5,0) -- (-4.5,0);
\draw (-3.5,0) -- (3.5,0);
\draw (-4.0,-0.5) node {\mbox{\fontsize{9}{0}\selectfont $-\frac{1}{\ell_0\ell_1\cdots}$}};
\draw[densely dotted] (-3.5,0) -- (-3.75,0);
\draw[densely dotted] (-4.25,0) -- (-4.5,0);
\draw[dotted] (-3.75,0) -- (-4.25,0);
\fill[white] (-4,0) circle (0.2em);
\draw (-4,0) circle (0.2em);

\draw[densely dotted] (8.5,0) -- (8.75,0);
\draw[densely dotted] (9.25,0) -- (9.5,0);
\draw[dotted] (8.75,0) -- (9.25,0);
\draw (9.0,-0.5) node {\mbox{\fontsize{9}{0}\selectfont $\frac{1}{\ell_0}+\frac{1}{\ell_0\ell_1}+\cdots$}};
\fill[white] (9,0) circle (0.2em);
\draw (9,0) circle (0.2em);

\draw[dotted] (0,2.25) -- (0,1.75);
\draw (-2,2) node {\mbox{\fontsize{9}{0}\selectfont $\frac{1}{\ell_0^2}+\frac{1}{(\ell_0\ell_1)^2}+\cdots$}};
\draw (0,2.5) -- (0,7);
\fill[white] (0,2) circle (0.2em);
\draw (0,2) circle (0.2em);

\draw (8,0.5) node {$\Uplambda(\mathbb T)$};
\draw[-latex, gray, thick] (7.3,0.55) -- (6,0.55);

\draw [decoration={brace,amplitude=0.7em,mirror},decorate,thick,gray] (-3.75,-3.3) -- (3.75,-3.3);
\draw (0,-4) node {$\I(\mathbb T)$};

\draw (0.5,6) node {$\mathbb T$};
\draw (11,0.5) node {$\mathbb T$};
\draw[-latex, gray, thick] (11.5,0.5) -- (12,0.5);
\draw[-latex, gray, thick] (0.5,6.5) -- (0.5,7);

\draw (0.5,-0.5) node {$0$};


\draw[densely dotted, domain=3.5:4.5] plot (\x, { ( -0.025*(\x-9) * (\x-9) + 2 ) });
\draw[densely dotted, domain=8.5:9] plot (\x, {  ( -0.025*(\x-9) * (\x-9) + 2 ) });
\draw[line width=0.1em, domain=-3.5:3.5] plot (\x, {  ( -0.025*(\x-9) * (\x-9) + 2 ) });
\draw[densely dotted, domain=-4.5:-3.5] plot (\x, {  ( -0.025*(\x-9) * (\x-9) + 2 ) });
\draw[line width=0.1em, domain=-5:-4.5] plot (\x, {  ( -0.025*(\x-9) * (\x-9) + 2 ) });
\draw[line width=0.1em, domain=4.5:8.5] plot (\x, {  ( -0.025*(\x-9) * (\x-9) + 2 ) });
\draw (-2.5,-2) node {$\ome$};
\fill[white] (9,2) circle (0.4em);
\draw (9,2) circle (0.4em);

\end{tikzpicture}
\caption{}\label{fig:omega}
\end{figure}

We need the following strengthening of Theorem~\ref{coram}, where 
the $\upo_{\rho}$ are as in that theorem:

\begin{theorem}\label{coramrho} Suppose the $H$-field $K$ is $\upo$-free. Then 
the sequence $(\upo_{\rho})$ has a
pseudolimit $\upo$ in an immediate $H$-field extension $K\<\upo\>$ 
of $K$ such that for any pseudolimit $a$ of 
$(\upo_{\rho})$ in any $H$-field 
extension of $K$ there is a unique isomorphism 
$K\<\upo\> \to K\<a\>$ over $K$ of ordered valued differential fields 
sending~$\upo$ to~$a$.
\end{theorem}

Using also a result from \cite{dagap}, this theorem has the following consequence:

\begin{cor}\label{prop:rho}
$\ome(\T)$ is not quantifier-free definable in the 
$\mathcal L'_{\Upl}$-struc\-ture $\T$.
\end{cor}

The next candidate of a language in which $\mathbb T$ might eliminate
quantifiers is
the extension $\mathcal L'_{\Upl,\Upo}$ of 
$\mathcal L'_{\Upl}$ by a unary predicate symbol 
$\Upo$, interpreted in $\mathbb T$ by~$\ome(\mathbb T)$.
At this stage we do not know of any obstruction to this possibility. 

Propositions~\ref{prop:NIP}, ~\ref{indconst}, and ~\ref{prop:2 analogue} remain true with 
$\mathcal L'_{\Upl,\Upo}$ replacing $\mathcal L'_{\I}$: 
the proofs of ~\ref{prop:NIP} and  ~\ref{indconst} go through because 
$\Upl(\T)$ and $\ome(\T)$ are convex subsets of $\T$, while the proof of ~\ref{prop:2 analogue} needs some further elaboration. 

We conclude that if
the theory of $\T$ as an $\mathcal L'_{\Upl,\Upo}$-structure 
admits elimination of quantifiers, then the $\T$-Conjecture from 
Section~\ref{sec:T-conjecture} holds, and 
$\T$ has properties (1), (2), (3) stated at the beginning of 
Section~\ref{sec:qf definability}.

\bigskip\noindent
[While this paper was under review, the statement after 
Corollary~\ref{prop:rho} that  ``we do not know of any obstruction to this possibility''  became obsolete. We now believe that the ``correct'' new primitive
that will enable us to get quantifier elimination for $\T$ is the
function $\T \to \T$ that maps each $g\in \ome(\T)$ to the unique $f\in \Lambda(\T)$
with $\ome(f)=g$, and sends each $g\in \T\setminus \ome(\T)$ to $0$; so
this function is the obvious partial inverse to $\ome$.]

\bibliographystyle{jflnat}

\begin{thebibliography}{10}
\expandafter\ifx\csname natexlab\endcsname\relax\def\natexlab#1{#1}\fi
\def\docolon{:}
\def\eatcomma#1{}
\def\junior{ Jr.}
\def\catchjrname#1,{#1 \one, \two,}
\def\onlyone#1{\gdef\oneletter{#1}}
\def\checkforperiod#1.#2*{\gdef\afterperiod{#2}}%
\def\reverseeditorname#1,#2,{\def\one{#1}\def\two{#2}%
\setbox0=\hbox{\expandafter\checkforperiod\one.*\expandafter\onlyone\two}%
\ifx\two\junior\let\go\catchjrname\else%
\let\go\relax%
\ifx\afterperiod\empty\oneletter. #1,\else#1,#2,\fi\fi\go}
\def\sphref#1#2{{\let\#=\docolon\xdef\one{#1}}\href{\one}{#2}}
\def\zhref#1,#2{{\let\#=\docolon\xdef\one{#1}}\href{\one}{#2}}
\expandafter\ifx\csname url\endcsname\relax
  \def\url#1{{\tt #1}}\fi
\newcommand{\enquote}[2]{``#1,''}

\bibitem{AvdD}
M. Aschenbrenner and L. van den Dries, {\em Closed
asymptotic couples,} J. Algebra {\bf 225} (2000), 309--358.

\bibitem{AvdD1} M. Aschenbrenner and L. van den Dries, \textit{$H$-fields and their Liouville extensions,} Math. Z. {\bf 242}  (2002), 543--588.

\bibitem{AvdD2} M. Aschenbrenner and L. van den Dries, \textit{Liouville closed $H$-fields,} J. Pure  Appl. Algebra {\bf 197} (2005), 83-139.

\bibitem{AvdD3} M. Aschenbrenner and L. van den Dries, \textit{Asymptotic differential algebra,} in: O. Costin, M. D. Kruskal, A. Macintyre (eds.), \textit{Analyzable Functions and Applications,}  49--85, Contemp. Math., vol. 373, Amer. Math. Soc., Providence, RI (2005).

\bibitem{dagap}
M. Aschenbrenner, L. van den Dries, and J. van der Hoeven, \textit{Differentially al\-ge\-braic gaps,} Selecta Math. {\bf 11} (2005), 247--280.

\bibitem{AK66}
J. Ax and S. Kochen, \textit{Diophantine problems over local fields, \textup{III}. Decidable fields, } Ann. of Math. (2) {\bf 83} (1966), 437--456. 

\bibitem{AK65}
J. Ax and S. Kochen, \textit{Diophantine problems over local fields, \textup{I},} Amer. J. Math. {\bf 87} (1965), 605--630.

\bibitem{Azgin-vdDries}
S. Azgin and L. van den Dries, \textit{Elementary theory of valued fields with a valua\-tion-preserving automorphism,} 
J. Inst. Math. Jussieu {\bf 10} (2011), no. 1, 1--35. 

\bibitem{BMS}
L. B\'elair, A. Macintyre, and T. Scanlon, {\it
Model theory of the Frobenius on the Witt vectors,} Amer. J. Math. {\bf 129} (2007), no. 3, 665--721. 

\bibitem{duBoisReymond}
P. du Bois-Reymond, {\em Ueber asymptotische Werthe, 
infinit\"are Approximationen und
infinit\"are Aufl\"osung von Gleichungen,} Math. Ann. {\bf 8} (1875),
362--414.

\bibitem{Borel} \'E. Borel, {\em M\'emoire sur les s\'eries divergentes,}
Ann. Sci. \'Ecole Norm. Sup. {\bf 16} (1899), 9--131.

\bibitem{Boshernitzan}
M. Boshernitzan, {\em New ``orders of infinity'',} J. Analyse Math. {\bf 41} (1982), 130--167.


\bibitem{vdD} L. van den Dries, \textit{Limit sets in o-minimal structures}, in: M. Edmundo et al. (eds.), \textit{Proceedings of the RAAG Summer School Lisbon 2003: O-minimal Structures,}  172--215, Lecture Notes in Real Algebraic and Analytic Geometry,
Cuvillier Verlag, G\"ottingen, 2005.

\bibitem{vdDL}
L. van den Dries and A. Lewenberg, \textit{$T$-convexity and tame extensions}, 
J. Symbolic Logic {\bf 60} (1995), no. 1, 74--102.

\bibitem{DMM1}
L. van den Dries, A. Macintyre, and D. Marker,  \textit{Logarithmic-exponential power series,} 
J. London Math. Soc. (2) {\bf 56} (1997), no. 3, 417--434. 

\bibitem{DMM}
L. van den Dries, A. Macintyre, and D. Marker, 
\textit{Logarithmic-exponential series,} Ann. Pure Appl. Logic
{\bf 111} (2001), 61--113.

\bibitem{DG} 
B. Dahn and P. G\"oring, \textit{Notes on exponential-logarithmic terms,} Fund. Math. {\bf 127} (1987), no. 1, 45--50.

\bibitem{DL}
J. Denef and L. Lipshitz, \textit{Power series solutions of algebraic differential equations,}
Math. Ann. {\bf 267} (1984), no. 2, 213--238. 


\bibitem{E}
J. \'Ecalle, \textit{Introduction aux Fonctions Analysables et Preuve
Constructive de la Conjecture de Dulac,} Actualit\'es Math\'ematiques, 
Hermann, Paris, 1992.

\bibitem{Edgar}
G. A. Edgar, \textit{Transseries for beginners,} 
Real Anal. Exchange {\bf 35} (2010), no. 2, 253--309.

\bibitem{Ersov}
Ju. L. Er{\v{s}}ov, \textit{On the elementary theory of maximal normed fields}, 
Soviet Math. Dokl. {\bf 6} (1965), 1390--1393.

\bibitem{Fornasiero}
A. Fornasiero, \textit{Tame structures and open cores,} preprint (2010).

\bibitem{GP} N. Guzy and F. Point, \textit{Topological differential fields,}
Ann. Pure Appl. Logic {\bf 161} (2010), no. 4, 570--598.

\bibitem{hardy} G. H. Hardy, {\em Orders of Infinity}, 2nd ed., Cambridge Univ. Press, Cambridge (1924).

\bibitem{Hille}
E. Hille, {\it Ordinary Differential Equations in the Complex Domain,}
Pure and Applied Mathematics, Wiley-Interscience, New York-London-Sydney, 1976.

\bibitem{vdH:PhD} J. van der Hoeven, {\em Asymptotique Automatique,} Th\`ese,
\'Ecole Polytechnique, Paris (1997).

\bibitem{vdH:noeth}
J. van der Hoeven, \emph{Operators on generalized power series}, Illinois J. Math.
  \textbf{45} (2001), no.~4, 1161--1190.

\bibitem{vdH}
J. van der Hoeven, {\em Transseries and Real Differential Algebra,}
Lecture Notes in Math., vol. 1888, Sprin\-ger-Verlag, Berlin, 2002.


\bibitem{vdH:hfsol} J. van der Hoeven, {\em Transserial Hardy fields},
 in: F. Cano et al. (eds.), \textit{Differential Equations and Singularities. 60 years of J.M. Aroca,}
453--487, Ast\'erisque {\bf 323}, S.M.F., 2009. 

\bibitem{HPP}
E. Hrushovski, K. Peterzil, A. Pillay, {\it Groups, measures, and the NIP,} J. Amer. Math. Soc. {\bf 21} (2008), no. 2, 563--596.


\bibitem{KKS} F.-V. Kuhlmann, S. Kuhlmann, and S. Shelah, \textit{Exponentiation
in power series fields,} Proc. Amer. Math. Soc. {\bf 125} (1997), 3177--3183.


\bibitem{Ressayre2}
J.-P. Ressayre, \textit{La th\'eorie des mod\`eles, et un petit probl\`eme de Hardy,} in: 
J.-P. Pier (ed.), \textit{Development of Mathematics 1950--2000,} 925--938,
Birkh\"auser Verlag, Basel, 2000.

\bibitem{Rosenlicht79} M. Rosenlicht, {\em On the value group of a
differential
valuation,} Amer. J. Math. {\bf 101} (1979), 258--266.

\bibitem{Rosenlicht80} M. Rosenlicht, {\em Differential valuations,}
Pacific J. Math. {\bf 86} (1980), 301--319.

\bibitem{Rosenlicht81} M. Rosenlicht, {\em On the value group of a
differential valuation,} II, Amer. J. Math. {\bf 103} (1981), 977--996.

\bibitem{Rosenlicht83}
M. Rosenlicht, \textit{The rank of a Hardy field,} Trans. Amer. Math. Soc. {\bf 280} (1983), no. 2, 659--671.
 
\bibitem{Rosenlicht95}
M. Rosenlicht, {\em Asymptotic solutions of ${Y}''={F}(x){Y}$}, J. Math. Anal. Appl.
{\bf 189} (1995), no.~3, 640--650.

\bibitem{Scanlon}
T. Scanlon, {\it A model complete theory of valued $D$-fields,} J. Symbolic Logic {\bf 65} (2000), no. 4, 1758--1784.

\bibitem{S1}
S. Shelah, \emph{Stability, the f.c.p., and superstability; model theoretic properties of formulas in first-order theories,}
Ann. Math. Logic {\bf 3} (1971), no. 3, 271--362.

\bibitem{Wilkie} A. Wilkie, {\em Model completeness results for expansions
of the ordered field of real numbers by restricted Pfaffian functions and
the exponential function,} J. Amer. Math. Soc. {\bf 9} (1996), 1051--1094.


\end{thebibliography}

\end{document}